\documentclass[11pt]{amsart} 

\usepackage{amsmath,amsthm}
\usepackage{amssymb}
\usepackage{amsfonts}
\usepackage{booktabs}
\usepackage{bbm}

\usepackage{enumitem}

\usepackage{graphicx}

\usepackage{xcolor,framed,tikz,float,hyperref,url}

\definecolor{shadecolor}{gray}{0.9}
\usetikzlibrary{arrows.meta,calc}

\usepackage{comment}

\newtheorem{theorem}{Theorem}[section]
\newtheorem{corollary}[theorem]{Corollary}
\newtheorem{lemma}[theorem]{Lemma}
\newtheorem{proposition}[theorem]{Proposition}

\theoremstyle{definition}
\newtheorem{definition}[theorem]{Definition}
\newtheorem{remark}[theorem]{Remark}

\newcommand{\B}{\mathcal{B}}
\newcommand{\R}{\mathbb{R}}

\newcommand{\U}{\mathcal{U}}
\newcommand{\Q}{\mathbb{Q}}
\newcommand{\AD}{\mathcal{H}}

\newcommand{\cE}{\mathcal{E}}

\newcommand{\Span}{\operatorname{span}_{\Q}}

\newcommand{\CQ}{\mathcal{C}_{\Q}}
\newcommand{\CbiQ}{\mathcal{C}_{\mathrm{bi}\Q}}

\begin{document}

\title[Cauchy-Hamel Continuity]{On the Cauchy-Hamel Continuity of Real Functions}

\author{Gabriel Istrate}
\address{
Faculty of Mathematics and Computer Science\\ 
University of Bucharest\\
Bucharest, Romania}
\email{gabriel.istrate@unibuc.ro} 

\date{}


\begin{abstract}

The classical Cauchy functional equation provides the simplest setting in
which algebraic structure and regularity come sharply apart: without a
regularity assumption, additive functions on $\mathbb{R}$ may be highly
discontinuous.  This raises a complementary question to the usual problem
of automatic continuity: rather than imposing weak regularity assumptions
that force an additive function to be linear, can ordinary continuity itself
be weakened in a natural way so that arbitrary additive functions become
continuous?

We investigate this question by comparing several extensions of ordinary
continuity.  Particular attention is given to $\mathbb{Q}$-continuity, a
nonstandard notion defined through continuous interpolation on rational
affine segments, and to a stronger bilateral version.  We
obtain structural characterizations of these notions in terms of rational affine lines; in
particular, bilateral $\mathbb{Q}$-continuity is equivalent to continuous
extendability of the restriction of the function from every affine line
$x+\mathbb{Q}h$.  We also compare these notions with two natural topological
constructions, the initial topology $\tau_{\AD}$ making all additive functions
continuous and a topology $\tau_{\Q}$ generated by rational directions.

The resulting classes display strong separation phenomena.  Assuming the
Axiom of Choice, neither $\mathbb{Q}$-continuity nor bilateral
$\mathbb{Q}$-continuity is the class of continuous functions associated with
any topology on $\mathbb{R}$.  Conversely, every $\tau_{\Q}$-open set is the
inverse image of an ordinary open set under a bilaterally
$\mathbb{Q}$-continuous function.  We further study Baire-category and
descriptive-set-theoretic properties of the associated topologies and show
that generalized polynomials provide a natural class of bilaterally
$\mathbb{Q}$-continuous functions.

The problem also has a classical motivation from the axiomatic foundations
of mechanics, where discontinuous additive functions arise when continuity
is removed from axiomatizations of the parallelogram rule for the composition
of forces.  Our results suggest bilateral $\mathbb{Q}$-continuity as a
candidate for a notion of \emph{Cauchy-Hamel continuity} and as a possible
starting point for a weak calculus compatible with arbitrary additive
functions.
\end{abstract}

\maketitle

\tableofcontents 

\section{Introduction}

Ordinary calculus is built upon the idea of approximating the slope of a function near a given point by a linear function - that is, by a \emph{continuous} additive function. 

However, linear functions are quite special within the set of \emph{additive functions} \cite{12} (sometimes called \emph{Cauchy-Hamel functions}\footnote{The terminology has changed: initially discontinuous additive functions were called \emph{Hamel functions} (e.g. in \cite{12}) but now this name is  reserved for a quite different notion, e.g. \cite{14}. Calling them \emph{Cauchy-Hamel functions} is a reasonable compromise.}): under the hypothesis that a \emph{Hamel basis} of $\R$ over $\Q$ exists\footnote{A well-ordering of \(\R\) yields a Hamel basis, whereas in ZF the existence of a Hamel basis does not imply that \(\R\) can be well-ordered; see \cite{3}.} all additive functions other than the linear ones are  discontinuous everywhere and take a dense set of values in every interval. 

An intriguing question that has already motivated (explicitly or implicitly) some research (e.g. \cite{4},\cite{11})\footnote{We asked precisely this question in our unpublished M.Sc. Thesis from 1994 \cite{9}.} is whether an "exotic" theory of differentiability exists that is based on locally approximating a function near a point by  \emph{(bijective) additive functions} (rather than linear ones). We also require that  this "exotic" theory be well-behaved enough to retain some of the flavor of classical calculus\footnote{There is a deeper foundational angle to this question: that of developing \emph{nonstandard models of classical mechanics}, based on replacing the parallelogram rule with a "wilder" version. To keep things tractable, we defer a complete treatment of this connection to subsequent papers, and only give a couple of details in the Conclusions section.}. To build such a theory \textbf{a first requirement is to construct a correspondingly "exotic" notion of continuity that makes all bijective additive functions continuous:} indeed, to express ordinary differentiability of a function $f$ at the point $x_0$ one needs to have: 
\begin{equation} 
f(x)=f(x_0)+(x-x_0)\cdot g(x)
\end{equation} 
where $g$ is a function that is continuous at $x_0$. The continuity of a differentiable $f$ at $x_0$ follows from this equation and the following facts: 
\begin{itemize} 
\item Function $H(x)=x$ is continuous at $x_0$. 
\item Continuity at $x_0$ is preserved by addition/subtraction of constants. 
\item The product of two continuous functions at $x_0$ (in our case $h(x)=x-x_0$ and $g$) is continuous at $x_0$. 

\end{itemize} 
If we want to extend this to differentiability  with respect to an arbitrary bijective additive function $H$ so that Cauchy-Hamel differentiability with respect to $H$ at $x_0$ implies "weak continuity" at that point, 
 one needs to write 
\begin{equation} 
f(x)=f(x_0)+[H(x)-H(x_0)]\cdot g(x)
\end{equation} 
and require the following properties: 
\begin{enumerate} 
\item \textbf{$H$ is "weakly continuous" at $x_0$.} 
\item "Weak continuity at $x_0$" is preserved by addition/subtraction of constants.
\item The product of a function that is "weakly continuous" at $x_0$ (in our case $H(\cdot)-H(x_0)$) and an ordinarily continuous function (in our case $g$) is "weakly continuous" at $x_0$. 
 
\end{enumerate} 

Several continuity notions considered in \cite{11} satisfy, indeed, the basic requirement that all additive functions are continuous; they are found, in the terminology of that paper, among the so-called \emph{type B properties}, meaning that there exist (ordinarily) discontinuous additive functions that are continuous in the sense of the given property.  Specifically, of the six type B properties considered in \cite{11}, four satisfy our requirement. On the other hand, in \cite{9} we had independently proposed yet another such notion we called \emph{$\Q$-continuity}\footnote{Here $\Q$ refers to the set of rational numbers, not to quasicontinuity.}.  
The results in \cite{9} were finally published in \cite{10} and were mildly encouraging: we showed, for instance, that $\Q$-continuous functions are uniform limits of Darboux functions. This is, of course, somewhat reminiscent of the fact that ordinary continuous functions have the Darboux property, a kind of regularity of the ordinary theory that we want to preserve in the Cauchy-Hamel world. 

This paper aims to draw attention to the problem of studying notions of continuity
and differentiability based on additive functions, and to address it from a principled, foundational perspective:  rather than simply proceeding with the study of  notions proposed in \cite{9,10}, we start with multiple plausible notions of continuity for additive functions, and compare them in order to assess which of these definitions is more well-behaved, and deserves the name of \emph{Cauchy-Hamel continuity}.

The eight definitions we consider are 
\begin{itemize}
\item[(i).] the four notions from \cite{11} that make every additive function continuous: they are \emph{graph continuity} \cite{20}, \emph{symmetry} \cite{19}, \emph{local almost continuity} \cite{21}, and \emph{local almost quasicontinuity} \cite{21}. 
\item[(ii).] two variations on our original notion of $\Q$-continuity from \cite{10}. 
\item[(iii).] two continuity notions based on new topologies that refine $\tau_{std}$, the standard topology
on $\R$:

\begin{itemize} 
\item $\tau_{\Q}$ is a topology on $\R$ inspired by our notion of $\Q$-continuity (which is not
based on a topology).

\item $\tau_{\AD}$ is the initial topology on $\R$ that makes all additive functions continuous.

\end{itemize} 
\end{itemize} 

The purpose of this paper is to compare these notions. 
Even though our ultimate goal is to define and study differentiability with respect to an arbitrary additive function, we will not deal in this paper  with such notions, deferring their study to a
subsequent paper.

\subsection{Regularity Tests, and Overview of our Results.}
 What are the regularity 
features of ordinary continuity that we aim to emulate in the Cauchy-Hamel world? There
are many possible answers to this question. We came up with the following (reasonable but, of course, nonexhaustive) list: 

\begin{enumerate} 
\item Ordinary continuity should be a special case of the new notion of continuity. In particular, if the new notion of continuity is based on a topology $\tau$, then $\tau$ should refine $\tau_{std}$. 

\item The new notion of continuity should make all additive functions continuous.

\item The new notion of continuity should coincide with ordinary continuity on the class of Baire 1 functions.

\item All continuous functions (with respect to the new notion) should
have (some weak version of) the Darboux property.

\item All other things being equal, we favor notions of continuity defined
using topologies. However, since many of the candidate continuity notions are \emph{not} topological, this is only a desirable feature, rather than a requirement. Failing this condition will not, in itself,  disqualify a notion of continuity. 

\item Topological notions should not recognize as open any "small" nonempty open sets (where "small" is interpreted in the ordinary sense of Lebesgue measure and Baire category).

\item If the new notion of continuity is based on a topology $\tau$, then $\tau$ should admit a nontrivial, well-behaved, descriptive set theory. In particular it should satisfy Baire’s theorem, and should not make all sets Borel.
\item Finally, \emph{if the continuity notion is not already disqualified by previous tests} then there should be natural classes of examples from the literature, other than the additive functions, that are continuous. 
\end{enumerate} 

During the research for this paper we had considered another test:
\begin{enumerate}[start=9] 
\item For ordinary continuous functions the sum of a continuous and a Darboux
function is the uniform limit of a sequence of Darboux functions. A similar statement should be true for "weak" continuity. 
\end{enumerate} 

It turned out that this last proposed test was inadequate for a trivial reason: it already fails for additive functions. We
further discuss this issue in Section~\ref{sec:8}.

The results of our tests are summarized in Table~\ref{table:1}. Here are some further
comments:

\begin{table}[h]
\begin{tabular}{|c|c|c|c|c|c|}
\toprule
 & $\stackrel{\mbox{Bilateral}}{\mathbb{Q}\mbox{-continuity}}$  & $\mathbb{Q}$-continuity & $\tau_{\Q}$-continuity & $\tau_H$-continuity & $\stackrel{\mbox{local almost}}{\mbox{continuity}}$ \\
\midrule
Test 1 & YES & YES & YES & YES & YES \\
\hline
Test 2 & YES & YES & YES & YES & YES \\
\hline
Test 3 & YES & YES & YES & YES & YES \\
\hline
Test 4 & YES & YES & \textcolor{red}{NO} & \textcolor{red}{NO} & \textcolor{red}{NO}\\
\hline
Test 5 & NO & NO & YES & YES & NO\\
\hline 
Test 6 & N/A & N/A & YES & YES & N/A\\
\hline 
Test 7 & N/A & N/A & \textcolor{red}{NO} & \textcolor{red}{NO} & N/A\\
\bottomrule
\end{tabular}
\vspace{0.2cm}
\caption{An overview of test results (we refer to Section~\ref{sec:3} for the formal definitions of these classes). Each such result may be subject to additional set-theoretic axioms, fully specified in the text.}
\label{table:1} 
\end{table}

\begin{itemize} 
\item All eight continuity notions we considered satisfy the basic requirements of extending ordinary continuity (Test \#1),
and making all additive functions continuous (Test \#2).

\item Three of the four concepts from \cite{11} fail Test \#3 and, consequently, will be removed from further consideration. For the same reason, these concepts are not listed in Table~\ref{table:1}, nor will they be discussed anymore. A fourth property from \cite{11}, local almost continuity,  and the remaining four properties pass Test \#3: none of these definitions creates more continuous
functions among the (ordinary) Baire 1 functions.

\item Only the two $\Q$-continuity classes pass Test \#4: every function in either class is a uniform limit of a sequence of
Darboux functions. $\tau_{\AD}$-continuity, $\tau_{\Q}$-continuity and local almost continuity do not have this property. For these latter three notions the counterexample we give seems to preclude any natural variation on the
Darboux property, hence excluding them from further consideration. 

\item We study some properties of the two (already discarded) topological properties. Both $\tau_{\AD}$-continuity and $\tau_{\Q}$-continuity pass Test \#6: no nonempty meager set and no nonempty set of
Lebesgue measure zero can be open in these topologies. On the other hand, under additional set-theoretic hypotheses $\tau_{\AD}$-continuity and $\tau_{\Q}$-continuity fail Test \#7: assuming that a Hamel basis exists
the reals are of first Baire category in each of the two topologies. On the other hand, assuming CH and AC, the two topologies are examples of \emph{$Q$-spaces}: all 
sets of reals
are $F_{\sigma}$ ($G_{\delta}$).

\item The remaining two
notions, $\Q$-continuity and its bilateral version, are seemingly not based on a topological concept. We show, in fact, that (assuming the Axiom of Choice) \emph{no topological
notion will do}: they are intrinsically not based on a topology.

\end{itemize} 

We also give other results that further characterize (most of) the five continuity notions in Table~1: 
\begin{itemize} 
\item  We show in Section~\ref{sec:4} that three of the five notions we consider can be
characterized by gluing subclasses of continuous functions along "rational
lines".
\item We compare in Section~\ref{sec:5} the various notions of continuity. The relations are depicted in Figure~\ref{fig-1}. Briefly, bilateral $\Q$-continuity is the most demanding notion, followed by $\Q$-continuity,  then $\tau_{\Q}$-continuity and finally local almost continuity. On the other hand $\tau_{H}$-continuity also implies $\tau_{\Q}$-continuity, but is incomparable with $\Q$-continuity, bilateral or not. 
\end{itemize} 

\begin{figure} 
\includegraphics[width=8cm]{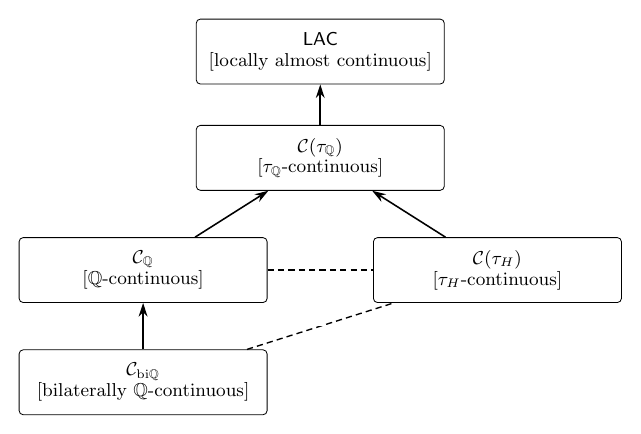}
\caption{Inclusions between classes in Table~1 (we refer to Section~\ref{sec:3} for the formal definitions of these classes). Arrows represent strict inclusions (assuming AC). Dashed lines represent incomparable classes.}
\label{fig-1}
\end{figure} 

In view of these results we feel that the two versions of $\Q$-continuity are the most
viable candidates for a notion of Cauchy-Hamel continuity. We further discuss this issue
in Section~\ref{sec:8}.

\subsection{A Note on Set-Theoretic Assumptions.} 

Several of our results require, in addition to the standard Zermelo-Fraenkel axioms, additional axioms such as Choice, the existence of a Hamel basis or (less frequently) the Continuum Hypothesis. 

Our choice is to work in ZF and prefix all the results we prove by the additional axioms that suffice in order  to prove the given result. 

\subsection{Paper Outline.} The outline of this paper is as follows: we first review some notions and results we will need in the sequel, and formally define the continuity notions that will be the candidates for the title of Cauchy-Hamel continuity. In Section~\ref{sec:fail} we show that three of the notions from \cite{11} already fail Test \#3 (this is why they are not listed in Table~\ref{table:1}). In Section~\ref{sec:4} we characterize three of the remaining notions by gluing specific kinds of functions across rational lines. These characterizations are then used to separate the surviving continuity notions. 

\section{Preliminaries} 

We will assume familiarity with general topology at the level of, say, \cite{26}, real analysis (see \cite{5}), and with basic set theory \cite{7}. For completeness we review below some of the notions we will need in the sequel. 

Whenever we will refer to the closure $\overline{A}$ of a set $A\subseteq \R$, the closure will be understood to be according to the standard topology on $
\R$, $\tau_{std}$ (unless mentioned otherwise). A similar convention will apply to the interior of set $A$, denoted $Int(A)$. We will use notation $\mathfrak{c}$ for $|\R|$. 

\begin{definition} Given a topological space $(X,\tau)$, 
a set $D\subseteq X$ is called \emph{discrete} if for every $x\in D$ there exists an open set $U\in \tau$ such that 
$D\cap U = \{x\}$. 

$(X,\tau)$ is called \emph{$\sigma$-discrete} if there is a sequence of subsets $D_1,D_2, \ldots, D_n, \ldots $ of $X$ such that 
\begin{enumerate} 
\item $X=\cup_{i\geq 1} D_i$. 
\item For every $i\geq 1$ the set $D_i$ is discrete. 
\end{enumerate} 
\end{definition} 

\begin{definition}
A set $A\subseteq\mathbb R$ is called \emph{preopen} if $A\subseteq\operatorname{Int}(\overline{A})$.
\end{definition}

Given a topology $\tau$ on $\R$, we will denote by $C(\tau)$ the set of continuous functions from $(\R,\tau)$ to $(\R,\tau_{std})$. We will also write $C(\R)$ instead of $C(\tau_{std})$.

\begin{definition}[Additive function]
A function $a:\R\to\R$ is called \emph{additive} if it is a solution of the Cauchy equation: 
\[
a(x+y)=a(x)+a(y)\qquad (x,y\in\R).
\]
Equivalently, $a$ is $\Q$-linear with respect to the $\Q$-vector space structure of $\R$. Let $\AD$ denote the class of additive functions $\R\to\R$. 
\end{definition}

\begin{definition}
Let $k\geq 1$. A function $
F\colon \mathbb{R}^k\to\mathbb{R}$ 
is called \emph{$k$-additive} if it is additive in each variable separately.
That is, for every $i\in\{1,\ldots,k\}$ and all
$x_1,\ldots,x_k,u,v\in\mathbb{R}$,
\[
\begin{aligned}
&F(x_1,\ldots,x_{i-1},u+v,x_{i+1},\ldots,x_k)\\
&\qquad =
F(x_1,\ldots,x_{i-1},u,x_{i+1},\ldots,x_k)
+
F(x_1,\ldots,x_{i-1},v,x_{i+1},\ldots,x_k).
\end{aligned}
\]

A $k$-additive function $F$ is called \emph{symmetric} if
$
F(x_{\sigma(1)},\ldots,x_{\sigma(k)})
=
F(x_1,\ldots,x_k)$ 
for every permutation $\sigma$ of $\{1,\ldots,k\}$.
\end{definition}
\begin{definition} 
Given a Hamel basis $\B$ and a real number $x$, define, for $b\in \B$, function $\pi_{\B,b}: \R\rightarrow \Q$ as follows: $\pi_{\B,i}(x)$ is the unique rational number such that 
\begin{equation} 
x= \sum_{j\in \B}\pi_{\B,j}(x)\cdot j 
\end{equation} 
is the linear expression of $x$ in basis $\B$. Also define 
\begin{equation} 
Supp_{\B}(x)=\{b\in \B: \pi_{\B,b}(x)\neq 0\}, 
\end{equation} 
$s_{B}(x)=|Supp_{B}(x)|$, and 
\begin{equation} 
S_{n,\B}=\{x\in \R: s_{B}(x)\leq n\}. 
\label{eq:sn}
\end{equation} 
Clearly, we have:
\begin{equation} 
\mathbb{R}=\bigcup_{n=1}^{\infty}S_{n,\B}. 
\end{equation} 
\end{definition}

\begin{definition}
Given $x\neq y$, define
\begin{equation} 
L(x,y)=\{q\cdot x + (1-q)\cdot y: q\in \Q\}
\end{equation} 
and 
\begin{equation} 
S(x,y)=\{q\cdot x + (1-q)\cdot y: q\in \Q\cap [0,1]\}
\end{equation} 

\end{definition}

\begin{definition}[(Strongly)  Darboux functions]

A function $f:\R\rightarrow \R$ is \emph{a Darboux function} iff for every $x<y$ such that $f(x)\neq f(y)$ and every $\xi$ strictly between $f(x)$ and $f(y)$, there exists $\eta \in (x,y)$ with $f(\eta)=\xi$.  Let $\mathcal{D}$
denote the class of Darboux functions. 

A function $d:\R\to\R$ is \emph{strongly Darboux} if for every nondegenerate interval
$I\subset\R$ one has $d(I)=\R$.
We denote the class of such functions by $D^*(\R)$.
\end{definition}

\begin{definition}[The class $\U$ {\cite{6}}]\label{def:U}
Let $\U$ denote the class of functions which are uniform limits of sequences of Darboux functions.
\end{definition}



Recall that $f$ is \emph{Baire $1$} if it is a pointwise limit of continuous functions.
We will use the following classical theorem:

\begin{theorem}[Baire's continuity theorem]\label{thm:baire-continuity}
If $f:\mathbb{R}\to\mathbb{R}$ is Baire $1$, then the set $
C(f):=\{x\in\mathbb{R}:\ f\text{ is continuous at }x\}$ 
is a dense $G_\delta$ subset of $\mathbb{R}$. In particular, $C(f)$ is dense in every nonempty open interval.
\end{theorem}

\begin{definition} A \emph{Q-space} \cite{25} is a topological space $(X,\tau)$ such that every subset of $X$ is $F_{\sigma}$ (equivalently $G_{\delta}$). \footnote{We will reserve the name \emph{$Q$-set space} \cite{23} for the setting when we additionally require that the space is not $\sigma$-discrete.}
\end{definition} 

\begin{definition} Given $h\neq 0$ and $f:\R\rightarrow \R$, \emph{the forward difference of order 1 of $f$ with respect to $h$} is the function $\Delta_{h}^{1}f:\R\rightarrow \R$, 
\begin{equation} 
\Delta_{h}^{1} f(x)=f(x+h)-f(x). 
\end{equation} 
For $n\geq 2$ we define \emph{the $n$-th fold forward difference of $f$ with respect to $h$} by
\begin{equation} 
\Delta_{h}^{n} f(x)=\Delta_{h}^{1} (\Delta_{h}^{n-1} f)(x).  
\end{equation} 
\end{definition}

\section{Candidate Notions for Cauchy-Hamel Continuity}
\label{sec:3} 

To make the paper self-contained, we now recall several concepts and results
from \cite{10,11}:

\begin{definition} 
A function \(f:\mathbb R\to\mathbb R\) is \emph{symmetric} at point \(x\in\mathbb R\) \cite{19} iff
\begin{equation} 
\lim_{h\to 0}\bigl(f(x+h)+f(x-h)-2f(x)\bigr)=0.
\end{equation} 
\end{definition} 

\begin{definition} 
A function \(f:\mathbb R\to\mathbb R\) is called \emph{locally almost continuous} \cite{21} if for every \(x\in\mathbb R\) and every neighborhood \(V\) of \(f(x)\), one has
\begin{equation} 
x\in \operatorname{Int}(\overline{f^{-1}(V))}.
\end{equation} 
\end{definition} 

\begin{definition} 
Function $f$ is called \emph{locally almost quasicontinuous} \cite{21} if for every \(x\in\mathbb R\) and every neighborhood \(V\) of \(f(x)\), one has
\begin{equation}
x\in \overline{\operatorname{Int}(\overline{f^{-1}(V)})}.
\end{equation} 
\end{definition}

\begin{definition} 
Function $f$ is \emph{graph continuous} \cite{20} if the closure of its graph contains the graph of a continuous function.
\end{definition} 

\begin{definition}[$\mathbb{Q}$-continuity \cite{10}]
Let $f:\mathbb{R}\rightarrow\mathbb{R}$ and $x\in\mathbb{R}.$ We say that \emph{$f$ is
$\mathbb{Q}$-continuous at $x$ from the left} if there exists $\epsilon>0$ such that for every $y, z$ with
$x-\epsilon<y<z\le x$ there exists a continuous function $f_{y,z}:[y,z]\rightarrow\mathbb{R}$ satisfying
\begin{equation}
f_{y,z}(t)=f(t) \quad \text{for all } t\in S(y,z).
\end{equation}
\label{def:qcont}
\end{definition}

$\mathbb{Q}$-continuity from the right is defined analogously (with $x\le y<z<x+\epsilon$), and
$f$ is $\mathbb{Q}$-continuous at $x$ if it is $\mathbb{Q}$-continuous at $x$ from both sides. We say $f$ is
$\mathbb{Q}$-continuous (everywhere) if it is $\mathbb{Q}$-continuous at every $x\in\mathbb{R}.$

\begin{remark}
Related notions of path continuity and path versions of the Darboux property
have been studied by Banaszewski and Marciniak
\cite{41,42,43,44}.
$\Q$-continuity is more rigid and is specifically adapted to the $\Q$-linear
structure of additive functions: continuous extendability is required along
every sufficiently local rational chord of the prescribed type rather than
along one suitably chosen path.
\end{remark} 

\begin{definition}
On the other hand we say that \emph{$f$ is bilaterally $\Q$-continuous at $x$} if
there exists $\epsilon>0$ such that for every $y, z$ with $x-\epsilon\le y<z\le x+\epsilon$ there exists
a continuous function $f_{y,z}:[y,z]\rightarrow\mathbb{R}$ satisfying
\begin{equation}
f_{y,z}(t)=f(t) \quad \text{for all } t\in S(y,z).
\end{equation}
\label{def:bilateral}
\end{definition}

\begin{theorem}[Closure and additive examples \cite{10}, Thm. 2.1]
If $f,g$ are $\mathbb{Q}$-continuous at $x_{0}$ (left/right/bilateral), then so are $f+g$, $\alpha f$ and $fg$ for any
$\alpha\in\mathbb{R}$. Moreover every additive function $H$ (i.e. $H(x+y)=H(x)+H(y))$ is
$\mathbb{Q}$-continuous everywhere.
\end{theorem}

\begin{theorem}[$QC\subseteq\mathcal{U}$ \cite{10}, Thm. 2.2(3)]
Every $\mathbb{Q}$-continuous function belongs
to $\mathcal{U}$.
\end{theorem}

The definition from \cite{10} is not the only possible version formalizing the intuitive idea that we want all additive functions to be continuous. Here are two topological alternatives:

\begin{definition}
Define $\tau_{\AD}$ to be the initial topology on $\mathbb{R}$ generated
by $\mathcal{H}$, i.e. the smallest topology such that each $H\in\AD$ is continuous as a map
$ H:(\mathbb{R},\tau_{\AD})\rightarrow(\mathbb{R},\tau_{std})$, 
where $\tau_{std}$ is the usual topology. Equivalently, a subbasis for $\tau_{\AD}$ is
$\mathcal{S}_{\AD}:=\{H^{-1}(U):H\in\AD,U\in\tau_{std}\}$.
\end{definition}

\begin{remark} 
The idea of changing the topology in order to regularize discontinuous
additive functions has appeared previously.  Bernardi
\cite{40} fixes a Hamel basis and identifies $\R$ with the
finite-support subspace of a product of copies of $\Q$, equipped with the
induced product topology.  This construction is basis-dependent and is
different from the intrinsic topology $\tau_{\mathcal H}$ considered here.
We also note that the claim in \cite{40} that every additive map is
continuous for the product-coordinate topology requires qualification:
continuity holds for finite linear combinations of coordinate projections,
whereas an arbitrary additive map may involve infinitely many coordinates.
The topology $\tau_{\AD}$ avoids this issue by being defined as the
initial topology of the entire family of additive maps.
\end{remark} 

\begin{definition}
Define a translation-invariant topology $\tau_{\Q}$ on
$\mathbb{R}$ as follows: 
A set $U\subseteq\mathbb{R}$ is $\tau_{\Q}$-open iff for every $a\in U$ and every $x\neq 0$ there is $\varepsilon(a,x)>0$ such that 
\begin{equation} 
a+qx\in U \mbox{ for all }q\in \Q\mbox{ with }0<q<\varepsilon(a,x)
\end{equation} 
\label{def:tauq}

It is easy to see that $\tau_{\Q}$ is a topology: closure under intersection follows by taking minima across direction-dependent radii. 

\end{definition}

\begin{remark} 
We have $\tau_{std}\subseteq \tau_{\AD}\subseteq \tau_{\Q}$. The first inclusion holds because the identity map is additive. We will prove the second inclusion later, specifically during the proof of Lemma~\ref{lemma:4.3} below. 
\end{remark} 

\begin{definition}
A function $f:\mathbb{R}\rightarrow\mathbb{R}$ is $\tau_{\Q}$-continuous if it is continuous as a
map $f:(\mathbb{R},\tau_{\Q})\rightarrow(\mathbb{R},\tau_{std})$. A function $f:\mathbb{R}\rightarrow\mathbb{R}$ is
$\tau_{H}$-continuous if it is continuous as a
map $f:(\mathbb{R},\tau_{H})\rightarrow(\mathbb{R},\tau_{std})$.
\end{definition}

\section{First Results} 
\label{sec:fail} 
Our first result shows that three of the four candidate notions from \cite{11} fail Test \#3:

\begin{theorem} For each of the following properties 
there exist Baire 1 functions \(f:\mathbb R\to\mathbb R\) having that property, yet not continuous in the ordinary sense: 
\begin{enumerate}
    \item \(f\) is graph continuous.
    \item \(f\) is symmetric. 
        \item \(f\) is locally almost quasicontinuous. 
        \end{enumerate}
\end{theorem}

\begin{proof}
We treat the three assertions separately.

\medskip
\noindent
\textbf{(1) Graph continuity.}
Consider \emph{Thomae's function} (see e.g. \cite{24})
\begin{equation} 
 t(x)=
 \begin{cases}
 1/q,& x=p/q\in\mathbb Q, p,q\in \mathbb{Z}, (p,q)=1, q\ge 1,\\
 0,& x\notin\mathbb Q.
 \end{cases}
\end{equation} 
It is known that \(t\) is continuous at every irrational point and discontinuous at every rational point. It is also Baire 1 (see e.g. \cite{22}). 

We show that \(t\) is graph continuous. Indeed, let
$g(x)=0$ for all $x\in\mathbb R$. Clearly \(g\) is continuous. We claim that the graph of \(g\) is contained in the closure of the graph of \(t\). Indeed, fix \(x\in\mathbb R\). Choose a sequence of irrational numbers \((x_n)\) converging to \(x\). Then
$t(x_n)=0=g(x_n)$ 
for every \(n\), so
$(x_n,t(x_n))=(x_n,0)\to (x,0)=(x,g(x))$. 
Hence every point of the graph of \(g\) belongs to the closure of the graph of \(t\). Therefore \(t\) is graph continuous.

\medskip
\noindent
\textbf{(2) Symmetric second-difference.}
Define
\begin{equation} 
 s(x)=
 \begin{cases}
 \sin(1/x),& x\ne 0,\\
 0,& x=0.
 \end{cases}
\end{equation} 
This function is not continuous at \(0\). It is Baire 1, since for example the functions
\begin{equation} 
 s_n(x)=
 \begin{cases}
 \sin(1/x),& |x|\ge 1/n,\\
 n x\sin n,& |x|<1/n,
 \end{cases}
\end{equation} 
are continuous on \(\mathbb R\) and converge pointwise to \(s\).

We now verify the symmetric second-difference condition. If \(x\ne 0\), then \(s\) is continuous at \(x\), so $s(x+h)+s(x-h)-2s(x)\to 0$ as \(h\to 0\).
At \(x=0\), for every \(h\ne 0\),
$
 s(h)+s(-h)-2s(0)=\sin(1/h)+\sin(-1/h)-0=0.$
Thus
$
 \lim_{h\to 0}\bigl(s(h)+s(-h)-2s(0)\bigr)=0.$
So \(s\) satisfies the symmetric second-difference condition at every point, yet is not continuous.

\medskip
\noindent
\textbf{(3) Local almost quasicontinuity.}
Consider the step function
\begin{equation} 
 h(x)=
 \begin{cases}
 0,& x<0,\\
 1,& x\ge 0.
 \end{cases}
\end{equation} 
This function is not continuous at \(0\). It is Baire 1, for instance because the continuous functions
\begin{equation} 
 h_n(x)=
 \begin{cases}
 0,& x\le -1/n,\\
 nx+1,& -1/n<x<0,\\
 1,& x\ge 0
 \end{cases}
\end{equation} 
converge pointwise to \(h\).

We prove that \(h\) is locally almost quasicontinuous. Let \(x\in\mathbb R\), and let \(V\) be an open neighborhood of \(h(x)\).

If \(x<0\), then \(h(x)=0\). Choosing an interval \(( -\varepsilon,\varepsilon)\subseteq V\), we get
$(-\infty,0)\subseteq h^{-1}(V)$, 
so some open neighborhood of \(x\) is contained in \(h^{-1}(V)\). Hence 
\begin{equation} 
 x\in \operatorname{Int}(\overline{h^{-1}(V)})\subseteq \overline{\operatorname{Int}(\overline{h^{-1}(V)})}.
\end{equation}  
The case \(x>0\) is analogous.

It remains to consider \(x=0\). Since \(h(0)=1\), choose \(\varepsilon>0\) such that
$(1-\varepsilon,1+\varepsilon)\subseteq V.$
Then $[0,\infty)\subseteq h^{-1}(V)\subseteq \overline{h^{-1}(V)}$, 
so $(0,\infty)\subseteq \operatorname{Int}(\overline{h^{-1}(V)}).$
Consequently, 
\begin{equation} 
0\in \overline{(0,\infty)}\subseteq \overline{\operatorname{Int}(\overline{h^{-1}(V)})}.
\end{equation} 
Thus \(h\) is locally almost quasicontinuous at \(0\) as well.

We have therefore exhibited Baire 1, discontinuous examples for all three notions.
\end{proof}

Next we give several results that establish inclusions between remaining candidate notions: 

\begin{theorem} 
Every $\tau_{\AD}$-continuous function is $\tau_{\Q}$-continuous. 
\end{theorem} 
\begin{proof} 
We first need the following simple

\begin{lemma} 
Every additive function $H:\R\rightarrow \R$ is $\tau_{\Q}$-continuous. 
\label{lemma:4.3}
\end{lemma} 
\begin{proof} 

Let \(V\subseteq\R\) be Euclidean open and let \(a\in H^{-1}(V)\). Fix a direction \(x\ne0\). Since \(V\) is open and the map \(q\mapsto H(a)+qH(x)\) is continuous as a real affine map, there is \(\varepsilon_x>0\) such that \(H(a)+qH(x)\in V\) whenever \(q\in\Q\) and \(0<q<\varepsilon_x\). By \(\Q\)-linearity, $H(a+qx)=H(a)+qH(x)$, 
so \(a+qx\in H^{-1}(V)\) for all such \(q\). This is exactly Definition~\ref{def:tauq} at \(a\). Since \(a\) was arbitrary, \(H^{-1}(V)\) is \(\tau_{\Q}\)-open.
\end{proof}
Since $\tau_{\AD}$ is the initial topology generated by the additive functions, $\tau_{\AD}\subseteq \tau_{\Q}$, 
so every $\tau_{\AD}$-continuous function is $\tau_{\Q}$-continuous. 
\end{proof} 

The following observation is a simple consequence of the definitions: 

\begin{theorem} The following statements are true: 
\begin{itemize} 
\item[(a). ] Every bilaterally $\Q$-continuous function $f:\R\rightarrow \R$ is $\Q$-continuous. 
\item[(b). ] Every $\Q$-continuous function is $\tau_{\Q}$-continuous. 
\end{itemize} 
\label{thm:inclusions}
\end{theorem} 
\begin{proof}
\begin{itemize} 
\item[(a). ] This is immediate from the definitions. Bilateral $\Q$-continuity requires the extension property for every interval $[y,z]$ with $x-\varepsilon\le y<z\le x+\varepsilon$. In particular, it holds whenever $x-\varepsilon<y<z\le x$ (left continuity) and whenever $x\le y<z<x+\varepsilon$ (right continuity). Hence every bilaterally $\Q$-continuous function is $\Q$-continuous.

\item[(b). ] Let $f$ be $\Q$-continuous, let $a\in\mathbb R$, and let $V$ be an ordinary open neighborhood of $f(a)$. Choose $r > 0$ such that $(f(a)-r,f(a)+r)\subseteq V$ and set $U=f^{-1}(V)$. We prove that $U$ is $\tau_{\Q}$-open at $a$.

Fix an arbitrary direction $h\neq 0$. By $\Q$-continuity at $a$, there exists $\delta>0$ such that every interval $[y,z]\subset(a-\delta,a]$ or $[a,z]\subset[a,a+\delta)$ admits a continuous extension agreeing with $f$ on the rational interpolation set.

Choose a rational number $q_0$ with $0<q_0<\delta/|h|$, and put $t=a+q_0h$. Let $\varphi$ be the corresponding continuous extension on the interval joining $a$ and $t$. Since $\varphi$ is continuous at $a$, there exists $\eta>0$ such that $|\varphi(u)-\varphi(a)|<r$ whenever $|u-a|<\eta$ and $u$ belongs to the interval.

Now, if $q\in\mathbb Q$ satisfies $0<q<\min\{q_0,\eta/|h|\}$, then 
\begin{equation} 
a+qh=(1-\frac{q}{q_0})a+\frac{q}{q_0}t
\end{equation} is a rational convex combination of the endpoints. Hence $f(a+qh)=\varphi(a+qh)\in V$.
Thus, for every direction $h\neq0$, there exists $\varepsilon_h>0$ such that
$a+qh\in U$ whenever $q\in\mathbb Q$ and $0<q<\varepsilon_h$.

This is exactly the defining neighborhood condition for $\tau_{\Q}$. Therefore $U$ is $\tau_{\Q}$-open at $a$, proving that $f$ is $\tau_{\Q}$-continuous.
\end{itemize} 
\end{proof} 

The next result is only slightly more complicated: 

\begin{theorem}
Every $\tau_{\Q}$-continuous function $f:\mathbb R\to\mathbb R$ is locally almost continuous.
\end{theorem}

\begin{proof}

We start with the following: 

\begin{lemma}
Every $\tau_{\Q}$-open subset of $\mathbb R$ is preopen in the ordinary Euclidean topology.
\end{lemma}

\begin{proof}
Let $U\subseteq\mathbb R$ be $\tau_{\Q}$-open, and fix $a\in U$. By the definition of the topology $\tau_Q$, the translate $U-a$ is a $\tau_{\Q}$-neighborhood of $0$. Applying the defining neighborhood condition to the direction $1$, there exists $\varepsilon_+>0$ such that $q\in U-a$ whenever $q\in\mathbb Q$ and $0<q<\varepsilon_+$. Applying the same condition to the direction $-1$, there exists $\varepsilon_->0$ such that $-q\in U-a$ whenever $q\in\mathbb Q$ and $0<q<\varepsilon_-$. Consequently, $a+q\in U$ for every rational $q$ with $0<q<\varepsilon_+$, and $a-q\in U$ for every rational $q$ with $0<q<\varepsilon_-$. Since the rational numbers are dense in $\mathbb R$, it follows that $(a-\varepsilon_-,a+\varepsilon_+)\subseteq\overline{U}$. Therefore $a\in\operatorname{Int}(\overline{U})$. As $a\in U$ was arbitrary, we conclude that $U\subseteq\operatorname{Int}(\overline{U})$, so $U$ is preopen.
\end{proof}

Now let $x\in\mathbb R$, and let $V$ be an ordinary open neighborhood of $f(x)$. Since $f$ is $\tau_{\Q}$-continuous, the inverse image $f^{-1}(V)$ is $\tau_{\Q}$-open. By the preceding lemma, $f^{-1}(V)$ is preopen in the ordinary topology, and hence $f^{-1}(V)\subseteq\operatorname{Int}(\overline{f^{-1}(V)})$. Because $x\in f^{-1}(V)$, we obtain $x\in\operatorname{Int}(\overline{f^{-1}(V)})$. This is precisely local almost continuity at $x$. Since $x$ was arbitrary, $f$ is locally almost continuous on $\mathbb R$.
\end{proof}

In \cite{10} it was proved that every additive function is $\Q$-continuous. A similar result is easily seen to hold for bilateral $\Q$-continuity. For completeness we include a formal statement and proof below: 

\begin{theorem} 
Every additive function $H:\R\rightarrow \R$ is bilaterally $\Q$-continuous. 
\end{theorem}
\begin{proof} 
Fix $x\in\mathbb R$. Since the argument is independent of $x$, any $\varepsilon>0$ will do.

Let $y<z$ satisfy $x-\varepsilon\le y<z\le x+\varepsilon$. Every point of $S(y,z)$ has the form
$t=\alpha y+(1-\alpha)z$ with $\alpha\in\mathbb Q\cap[0,1]$.

Define
$g_{y,z}(t)=\frac{z-t}{z-y}H(y)+\frac{t-y}{z-y}H(z)$
for $t\in[y,z]$. This is an affine function of $t$, hence continuous.

If $t=\alpha y+(1-\alpha)z$ with $\alpha\in\mathbb Q$, then by additivity,
\begin{equation} 
H(t)=H(\alpha y+(1-\alpha)z)=\alpha H(y)+(1-\alpha)H(z)=g_{y,z}(t),
\end{equation} 
because every additive function is $\mathbb Q$-linear.

Thus $g_{y,z}$ agrees with $H$ on $S(y,z)$. Since this holds for every interval $[y,z]$ around $x$, $H$ is bilaterally $\Q$-continuous.

\end{proof}

\section{Weak Continuity Properties of Baire 1 functions} 

Complementing the results in the previous section, we show next that the remaining candidates for weak continuity pass Test \#3:  

\begin{theorem} 
A Baire 1 local almost continuous function is continuous. 
\end{theorem} 

\begin{proof}

The proof applies the following simple principle: 
\begin{lemma} 
If $f$ is continuous at $c$, $V$ is open in $\tau_{std}$  and $c\in \overline{f^{-1}(V)}$ then $f(c)\in \overline{V}$. 
\label{lemma:cont}
\end{lemma} 
\begin{proof} 
Indeed, if $f(c)\notin\overline{V}$, then the open set
$\mathbb{R}\setminus\overline{V}$ is a neighborhood of $f(c)$. By
continuity of $f$ at $c$, there is a neighborhood $U$ of $c$ such that $f(U)\subseteq \mathbb{R}\setminus\overline{V}$.  
Consequently $U\cap f^{-1}(V)=\varnothing$, contradicting
$c\in\overline{f^{-1}(V)}$.

\end{proof} 

Since $f$ is Baire $1$, Theorem~\ref{thm:baire-continuity}  implies that
$C(f)$ is a dense $G_\delta$ subset of $\mathbb{R}$. In particular,
$C(f)$ is dense in every nonempty open interval.

Fix $x\in\mathbb{R}$ and $\varepsilon>0$. Set
$
    V_x:=\left(f(x)-\frac{\varepsilon}{3},
               f(x)+\frac{\varepsilon}{3}\right)$ and 
    $A_x:=f^{-1}(V_x).$
By local almost continuity at $x$,
$x\in \operatorname{Int}(\overline{A_x}).$
Hence there exists a nonempty open interval $I$ containing $x$ such that $
    I\subseteq \overline{A_x}.$
We claim that $f(I)\subseteq \bigl(f(x)-\varepsilon,f(x)+\varepsilon\bigr)$. 

Suppose, toward a contradiction, that there exists $y\in I$ such that $
    |f(y)-f(x)|\geq \varepsilon.$ 
Define $V_y:=\left(f(y)-\frac{\varepsilon}{3},
               f(y)+\frac{\varepsilon}{3}\right)$
   and 
    $A_y:=f^{-1}(V_y).$
Local almost continuity at $y$ gives
 $y\in \operatorname{Int}(\overline{A_y}).$
Since $y\in I$ and $I$ is open, the set
$J:=I\cap \operatorname{Int}(\overline{A_y})$ 
is a nonempty open neighborhood of $y$. Since $C(f)$ is dense in every
nonempty open interval, there exists $
    c\in J\cap C(f).$ 
Thus $f$ is continuous at $c$, and $
    c\in I\subseteq \overline{A_x}$ and $c\in \overline{A_y}.$

Applying Lemma~\ref{lemma:cont} first to $V_x$ and then to $V_y$, we obtain
\begin{equation} 
    f(c)\in
    \left[f(x)-\frac{\varepsilon}{3},
          f(x)+\frac{\varepsilon}{3}\right]\mbox{ and } 
    f(c)\in
    \left[f(y)-\frac{\varepsilon}{3},
          f(y)+\frac{\varepsilon}{3}\right].
\end{equation}  
However, these two closed intervals are disjoint. Indeed, the distance
between their centers is at least $\varepsilon$, whereas the sum of their
radii is $
    \frac{\varepsilon}{3}+\frac{\varepsilon}{3}
    =\frac{2\varepsilon}{3}<\varepsilon.$
This contradiction proves the claim.

Therefore, for every $t\in I$, $|f(t)-f(x)|<\varepsilon.$
Thus $f$ is continuous at $x$. Since $x\in\mathbb{R}$ was arbitrary,
$f$ is continuous on $\mathbb{R}$.
\end{proof}

\section{Darboux-type properties of candidate notions}

One desired constraint we have on our notion of Cauchy-Hamel continuity is that some weak version of the ``continuous $\Rightarrow$ Darboux'' result continues to hold. A baseline illustration of this fact is the notion of $\mathbb{Q}$-continuity: In \cite{10} it was proved that $\mathbb{Q}$-continuous functions are the uniform limit of a sequence of Darboux functions. 

The next result implies the fact that the theorem from \cite{10} cannot be lifted to larger classes, such as that of $\tau_{\Q}$-continuous functions:

\begin{theorem}
\label{prop:6.4}
Assuming a Hamel basis exists, then 
$C(\tau_{\AD})\not\subset \mathcal{U}$; in particular, there exists a $\tau_{\AD}$-continuous function which is not in $\mathcal{U}$.
\end{theorem}

\begin{proof}
Choose a discontinuous additive non-surjective function $H:\mathbb{R}\rightarrow\mathbb{R}$. This can be done as follows: if $\B$ is a Hamel basis and $b_0\in \B$, define 
\begin{equation} 
H(b_0)=1, H(b)=0 (b\in \B\setminus \{b_0\})
\end{equation} 
and extend $H$ by linearity. Then $H(\R)=\Q$, hence $H$ is nonsurjective, nonzero, and discontinuous. 

Fix $\alpha\in\mathbb{R}\setminus H(\mathbb{R})$ and define $A:=H^{-1}((-\infty,\alpha))$. Since $(-\infty, \alpha)$ and $(\alpha, \infty)$ are Euclidean open and $\alpha \notin H(\mathbb{R})$ we have $\mathbb{R}\setminus A=H^{-1}((\alpha,\infty))$.

Therefore both $A$ and $\mathbb{R}\setminus A$ are $\tau_{\AD}$-open, i.e. $A$ is $\tau_{\AD}$-clopen. Hence the indicator function $f:=\mathbbm{1}_{A}$ is continuous as a map $(\mathbb{R},\tau_{\AD})\rightarrow(\mathbb{R},\tau_{std})$.

On the other hand, pick any $a\in A$ and any $b\notin A$. Then $f(a)=1$ and $f(b)=0$ but $f$ takes no value in $(0,1)$.

However, we know that $\mathcal{U}\subseteq\mathcal{U}_{0}$, the family of functions such that for every $a<b$ with $f(a)\ne f(b)$, the set $f((a,b))$ is dense in $(f(a),f(b))$ \cite{6}. But it is clear that $f\notin\mathcal{U}_{0}$. Consequently $f\notin\mathcal{U}$, proving $C(\tau_{\AD})\not\subseteq\U$.
\end{proof}

\begin{remark} 
Since $C(\tau_{\AD})\subseteq C(\tau_{Q})\subseteq LAC$, the same function is a witness for the larger classes containing functions not in $\U$ as well. 
\end{remark}

\section{Properties of the two topologies}

In this section we study the two topologies $\tau_{\Q}$ and $\tau_{\AD}$. We start with a desirable property that these two topologies share: 

\begin{theorem} 
Every nonempty \(\tau_{\Q}\)-open subset of \(\R\) is nonmeager and has positive outer Lebesgue measure. Consequently the same is true for every nonempty \(\tau_{\AD}\)-open subset.

In other words, the two topologies pass Test \#6. 
\end{theorem}
\begin{proof} 
Let \(U\ne\varnothing\) be \(\tau_\Q\)-open and choose \(a\in U\). For every \(x\ne0\), the definition of \(\tau_\Q\) supplies a positive rational \(q_x\) such that \(a+q_xx\in U\). Hence
\[
  \R=\bigcup_{q\in\Q_{>0}}q^{-1}(U-a).
\]
If \(U\) were meager, then every set \(q^{-1}(U-a)\) would be meager and the displayed countable union would make \(\R\) meager in its Euclidean topology, a contradiction. If \(U\) had outer measure zero, every set in the union would have outer measure zero, again a contradiction. Finally, \(\tau_{\AD}\subseteq\tau_\Q\), so every \(\tau_{\AD}\)-open set is \(\tau_\Q\)-open. 
\end{proof}

Next we highlight a couple of undesirable properties of the two topologies. The most serious one is described in the next section: with either of the two topologies, $\mathbb{R}$ is of first Baire category in itself. Yet another such paradoxical property is described in Subsection~\ref{sec:qspace}: assuming the Continuum Hypothesis the two topologies are $Q$-spaces. We close this section with two paradoxical results for the sets of rational/irrational numbers.

\subsection{The failure of Baire's theorem}

\begin{theorem}
\label{thm:7.1}
Assuming a Hamel basis exists, the space $(\mathbb{R}, \tau_{\Q})$ is meager (of first category).
\end{theorem}

\begin{proof}
Fix a Hamel basis $\mathcal{B}$ of $\mathbb{R}$ over $\mathbb{Q}$.  The following result implies the desired conclusion, that $(\mathbb{R},\tau_{\Q})$ is meager:

\begin{proposition}
\label{prop:7.2}
For every $n\in\mathbb{N}$ the set $S_{n,B}$ is $\tau_{\Q}$-closed and has empty $\tau_{\Q}$-interior. In particular, $S_{n,B}$ is $\tau_{\Q}$-nowhere dense.
\end{proposition}

\begin{proof}
First, we note that:

\begin{lemma} 
Let $x\in\mathbb{R}$ and let $b\in\mathcal{B}$ have nonzero coefficient $\alpha$ in the Hamel expansion of $x$. Let $y\in\mathbb{R}$ have coefficient $\beta$ at $b$. If $q\in\mathbb{Q}$ satisfies $|q\beta|<|\alpha|/2$ then the coefficient of $b$ in $x+qy$ is nonzero.
\end{lemma} 

\begin{proof}
The coefficient of $b$ in $x+qy$ equals $\alpha+q\beta$ and $|\alpha+q\beta|\ge|\alpha|-|q\beta|>|\alpha|/2>0$.
\end{proof}

Using this, we prove that $S_{n,\B}$ has the desired properties:
\begin{lemma} 
$\R\setminus S_{n,\B}$ is a $\tau_{\Q}$-open set. Hence $S_{n,\B}$ is $\tau_{\Q}$-closed.
\end{lemma} 
\begin{proof} 
Fix $n$ and let $x\notin S_{n,\B}$, so $s_\B(x)\ge n+1$. Choose distinct basis elements $b_{1},...,b_{n+1}\in\mathcal{B}$ occurring in the Hamel expansion of $x$ with nonzero coefficients $\alpha_{1},...,\alpha_{n+1}\in\mathbb{Q}\setminus\{0\}$.

For each $y\in\mathbb{R}\setminus\{0\}$, write its Hamel expansion and let $\beta_{i}(y)=\pi_{\B,b_i}(y)\in\mathbb{Q}$ denote the coefficient of $b_{i}$ in $y$. Define
\begin{equation} 
\epsilon_{y}:=\min\left(1,\min_{\substack{1\le i\le n+1 \\ \beta_{i}(y)\ne0}}\frac{|\alpha_{i}|}{2|\beta_{i}(y)|}\right)>0, \end{equation} 
with the convention that the inner minimum over an empty index set is $+\infty$ (so then $\epsilon_{y}=1$).

Choose $x\not\in S_{n,\B}$. For every direction $y\neq 0$, the chosen $\varepsilon_y>0$ yields 
\begin{equation} 
x+qy\not\in S_{n,\B}\mbox{ for }q\in \Q, 0<q<\varepsilon_{y}
\end{equation} 
\end{proof} 

So $S_{n,\B}$ is $\tau_{\Q}$-closed.

\begin{lemma} $S_{n,\B}$ has empty interior in $\tau_{\Q}$. 
\end{lemma} 
\begin{proof} Fix $x\in S_{n,\B}$ and let $U$ be any $\tau_{\Q}$-open neighborhood of $x$. 

Choose distinct basis elements $c_1,c_2, \ldots, c_{n+1}\not\in Supp_\B(x)$ and let $b=c_1+\ldots + c_{n+1}$. 

By definition of $\tau_{\Q}$, there exists $\epsilon>0$ such that $x+qb\in U$ for all rationals $q$ with $0<q<\epsilon$. For any such $q\ne0$, the Hamel expansion of $x+qb$ contains the support of $x$ together with $c_{1},\ldots c_{n+1}$, hence $s_\B(x+qb)\ge s_\B(x)+n+1$. In particular, choosing $q$ small but nonzero gives a point of $U$ outside $S_{n,\B}$. Thus $U\not\subseteq S_{n,\B}$ so $S_{n,\B}$ has empty $\tau_{\Q}$-interior.
\end{proof} 
\end{proof}
\renewcommand{\qedsymbol}{}
\end{proof}

Similarly for $\tau_{\AD}$ we have: 

\begin{theorem}
\label{thm:7.3}
Assuming a Hamel basis exists, the space $(\mathbb{R}, \tau_{\AD})$ is meager (of first category).
\end{theorem}

\begin{proof}
The proof strategy is similar, and relies on the following:

\begin{lemma}
\label{lem:7.4}
For every $n\in\mathbb{N}$ the set $S_{n,\B}$ is $\tau_{\AD}$-nowhere dense. Consequently, $(\mathbb{R},\tau_{\AD})$ is meager.
\end{lemma}

\begin{proof}
We show that $S_{n,\B}$ is $\tau_{\AD}$-closed and has empty $\tau_{\AD}$-interior. For each $b\in\mathcal{B}$ function $\pi_{b}$ is additive, hence $\tau_{H}$-continuous.
\begin{lemma} 
$S_{n,\B}$ is $\tau_{\AD}$-closed. 
\end{lemma} 
\begin{proof} 
Let $x\notin S_{n,\B}$ so $s_\B(x)\ge n+1$. Choose distinct basis elements $b_{1},...,b_{n+1}\in\mathcal{B}$ which occur in the Hamel expansion of $x$ with nonzero coefficients $\alpha_{1},...,\alpha_{n+1}\in\mathbb{Q}\setminus\{0\}$, i.e. $\pi_{\B,b_{i}}(x)=\alpha_{i}$. Consider the set
\begin{equation} 
U_{x}:=\bigcap_{i=1}^{n+1}\pi_{b_{i}}^{-1}\left(\left(\alpha_{i}-\frac{|\alpha_{i}|}{2},\alpha_{i}+\frac{|\alpha_{i}|}{2}\right)\right). \end{equation}
Each interval $(\alpha_{i}-\frac{|\alpha_{i}|}{2},\alpha_{i}+\frac{|\alpha_{i}|}{2})$ is Euclidean open and contains $\alpha_{i}$, and each $\pi_{\B,b}$ is $\tau_{\AD}$-continuous, hence $U_{x}$ is $\tau_{H}$-open and $x\in U_{x}$.

Moreover, if $u\in U_{x}$ then $\pi_{\B,b_{i}}(u)\ne0$ for all $i$, so $b_{1},...,b_{n+1}$ belong to the Hamel support of $u$, and therefore $s_\B(u)\ge n+1$, i.e. $u\notin S_{n,\B}$. Thus $U_{x}\subseteq\mathbb{R}\setminus S_{n,\B}$. This shows that for each $x\notin S_{n,\B}$ there exists a $\tau_{\AD}$-open neighborhood contained in $\mathbb{R}\setminus S_{n,\B}$, so $\mathbb{R}\setminus S_{n,\B}$ is $\tau_{\AD}$-open and $S_{n,\B}$ is $\tau_{\AD}$-closed.
\end{proof} 
\begin{lemma} 
$S_{n,\B}$ has empty $\tau_{\AD}$-interior.
\end{lemma}  
\begin{proof} 
Fix $x\in S_{n,\B}$ and let $V$ be any $\tau_{\AD}$ neighborhood of $x$. By translation invariance of $\tau_{\AD}$, there exists a $\tau_{\AD}$-neighborhood $W$ of $0$ with $x+W\subseteq V$. By definition of $\tau_{\AD}$, we may assume that 
\begin{equation}  W=\bigcap_{j=1}^{m}H_{j}^{-1}((-\epsilon_{j},\epsilon_{j})) \end{equation} 
for some additive functions $H_{1},...,H_{m}$ and some $\epsilon_{1},...,\epsilon_{m}>0$. Choose distinct basis elements $c_{1},...,c_{n+1}\in\mathcal{B}$ which do not occur in the Hamel expansion of $x$. We construct $w=\sum_{i=1}^{n+1}q_{i}c_{i}\in W$ with each $q_{i}\in\mathbb{Q}\setminus\{0\}$ by choosing the $q_{i}$ inductively.

Assume $q_{1},...,q_{i-1}$ are chosen and define $w_{i-1}:=\sum_{r=1}^{i-1}q_{r}c_{r}$. For each $j\in\{1,...,m\}$, set
\begin{equation} 
\delta_{j,i}:=\frac{\epsilon_{j}}{2^{i+1}}. 
\end{equation} 
Since $H_{j}$ is additive, $H_{j}(qc_{i})=q~H_{j}(c_{i})$ for $q\in\mathbb{Q}$, so by choosing $q_{i}\in\mathbb{Q}\setminus\{0\}$ sufficiently small we can ensure $|H_{j}(q_{i}c_{i})|<\delta_{j,i} \quad (j=1,...,m). $ 
Let $w_{i}:=w_{i-1}+q_{i}c_{i}$. Then for each $j$,$ |H_{j}(w_{i})|\le|H_{j}(w_{i-1})|+|H_{j}(q_{i}c_{i})|.$ 
Iterating this estimate gives, for $w:=w_{n+1}$,
\begin{equation} 
|H_{j}(w)|\le\sum_{i=1}^{n+1}|H_{j}(q_{i}c_{i})|<\sum_{i=1}^{n+1}\frac{\epsilon_{j}}{2^{i+1}}<\epsilon_{j}. \end{equation} 
Hence $w\in W$, so $x+w\in x+W\subseteq V$.

On the other hand, $w$ involves the $n+1$ basis elements $c_{1},...,c_{n+1}$ which are not in the Hamel support of $x$, so all those $c_{i}$ appear with nonzero coefficients in the Hamel expansion of $x+w$. Therefore $s_\B(x+w)\ge n+1$ and $x+w\notin S_{n,\B}$. This shows that every $\tau_{\AD}$-neighborhood $V$ of $x$ meets $\mathbb{R}\setminus S_{n,\B}$, so $S_{n,\B}$ has empty $\tau_{\AD}$-interior.
\end{proof} 
Since $S_{n,\B}$ is $\tau_{\AD}$-closed and has empty $\tau_{\AD}$-interior, it is $\tau_{\AD}$-nowhere dense. 
\end{proof}
\renewcommand{\qedsymbol}{}
\end{proof}

\subsection{The two topologies are $Q$-spaces.}
\label{sec:qspace} 
The following result gives yet another pathological property of the two topologies:

\begin{theorem}
\label{thm:7.5}
Assume CH and AC. Then
\begin{enumerate}
    \item[(a)] every subset of $\mathbb{R}$ is an $F_{\sigma}$ set in the topology $\tau_{\Q}$.
    \item[(b)] every subset of $\mathbb{R}$ is an $F_{\sigma}$ set in the topology $\tau_{\AD}$.
\end{enumerate}
\end{theorem}

\begin{remark} 
Balogh \cite{23} constructed a $Q$-set space in ZFC.  We note here that his construction can be adapted to show that there are topologies $\tau$ refining $\tau_{std}$ on $\R$ so that $(\R,\tau)$ is a $Q$-set space\footnote{To preserve the thematic consistency of this paper we reserve the justification of this observation to a subsequent note.}. On the other hand $(\R,\tau_{\Q})$ and $(\R,\tau_{\AD})$ are \emph{not} $Q$-set spaces, as they are $\sigma$-discrete: Lemmas~(\ref{lem:7.6}) and~(\ref{lem:7.7}) below prove that sets $L_n$ are discrete (because each of its subsets is closed).
\end{remark} 

\begin{proof}
We prove the theorem via the following two lemmas:

\begin{lemma}
\label{lem:7.6}
Assume AC. 
Let $L\subseteq\mathbb{R}$ be a $\mathbb{Q}$-linearly independent set, and let $A\subseteq L$. Then $A$ is $\tau_{\Q}$-closed.
\end{lemma}

\begin{proof}
If $L=\emptyset$ the result is trivial, so assume that $L\neq \emptyset$. 

We show that $\mathbb{R}\setminus A$ is $\tau_{\Q}$-open. Fix $x\in\mathbb{R}\setminus A$ and $y\in\mathbb{R}\setminus\{0\}$. Choose a Hamel basis $\mathcal{B}$ of $\mathbb{R}$ over $\mathbb{Q}$, and let
$F:=Supp_{\B}(x)\cup Supp_{\B}(y).$ 
Then $F$ is finite, and for every $q\in\mathbb{Q}$ we have $x+qy\in span_{\mathbb{Q}}(F).$ 
Hence, if $x+qy\in A\subseteq L$ then necessarily
$x+qy\in L\cap span_{\mathbb{Q}}(F).$ 
Now $span_{\mathbb{Q}}(F)$ is a finite-dimensional $\mathbb{Q}$-vector space, while $L$ is $\mathbb{Q}$-linearly independent. Therefore $L\cap span_{\mathbb{Q}}(F)$ is finite.

For each $z\in L\cap span_{\mathbb{Q}}(F)$, the equation $x+qy=z$ 
has at most one solution $q\in\mathbb{Q}$, because $y\ne0$. It follows that the set
$E_{x,y}:=\{q\in\mathbb{Q}_{>0}:x+qy\in A\}$ 
is finite. Choose $\epsilon_{x,y}>0$ such that
$0<q<\epsilon_{x,y}, \ q\in\mathbb{Q}\Rightarrow q\notin E_{x,y}.$ Equivalently, $0<q<\epsilon_{x,y}, \ q\in\mathbb{Q}\Rightarrow x+qy\notin A.$ 
Since this is true for every $y\ne0$, the defining neighborhood criterion for $\tau_{\Q}$ shows that $\mathbb{R}\setminus A$ is $\tau_{\Q}$-open at $x$. As $x\in\mathbb{R}\setminus A$ was arbitrary, $\mathbb{R}\setminus A$ is $\tau_{\Q}$-open, so $A$ is $\tau_{\Q}$-closed.
\end{proof}

\begin{lemma}
\label{lem:7.7}
Assume AC. 
Let $L\subseteq\mathbb{R}$ be $\mathbb{Q}$-linearly independent. Then $L$ is closed in the topology $\tau_{\AD}$.
\end{lemma}

\begin{proof}
If $L=\emptyset$ the result is trivial, so assume that $L\neq \emptyset$. 

Recall that $\tau_{\AD}$ is the initial topology with respect to the family of all additive maps $H:\mathbb{R}\rightarrow\mathbb{R}$. Equivalently, a subbasis of $\tau_{\AD}$ is formed by the sets
$H^{-1}(U), \quad H:\mathbb{R}\rightarrow\mathbb{R} \text{ additive, } U\subseteq\mathbb{R} \text{ open.}$
Thus it is enough to prove that for every $x\in\mathbb{R}\setminus L$ there exist an additive map $H$ and an open interval $I\subseteq\mathbb{R}$ such that $H(x)\in I \text{ and } H(L)\cap I=\emptyset. $
Then $H^{-1}(I)$ is a $\tau_{\AD}$-open neighborhood of $x$ disjoint from $L$.

Fix $x\in\mathbb{R}\setminus L$. 

If $x=0$ then, because $L$ is linearly independent, define a $\Q$-linear functional on $Span_{\Q}(L)$ as follows: $\phi(l)=1$ for $l\in L$, then extend $\phi$ by additivity to $H:\R\rightarrow \R$. Then $H(0)=0$ and $H(L)=\{1\}$, so $H^{-1}((-1/2,1/2))$ is a $\tau_{\AD}$-open neighborhood of 0 disjoint from $L$. 

Next, we deal with the case $x\neq 0$. We distinguish two subcases: 

\begin{itemize} 
\item[-] \noindent \textbf{Case 1:} $x\notin Span_{\mathbb{Q}}(L)$. \\
Then $L\cup\{x\}$ is still $\mathbb{Q}$-linearly independent. Define a $\mathbb{Q}$-linear functional $\varphi:Span_{\mathbb{Q}}(L\cup\{x\})\rightarrow\mathbb{R}$ by
$\varphi(l)=0 \ (l\in L), \varphi(x)=1.$ 
This is well defined because $L\cup\{x\}$ is linearly independent over $\mathbb{Q}$. Extend to an additive map $H:\mathbb{R}\rightarrow\mathbb{R}$ by extending a basis of $Span_{\mathbb{Q}}(L\cup\{x\})$ to a Hamel basis of $\mathbb{R}$. Then $H(L)=\{0\}$, $H(x)=1$. Therefore, with $I=(1/2,3/2)$, we have $x\in H^{-1}(I)$ and $H^{-1}(I)\cap L=\emptyset$. So $x$ has a $\tau_{\AD}$-open neighborhood disjoint from $L$.

\item[-] \noindent \textbf{Case 2:} $x\in Span_{\mathbb{Q}}(L)$. \\
Since $L$ is linearly independent and $x\notin L$, there exist distinct elements $l_{1},...,l_{n}\in L$ and nonzero rationals $q_{1},...,q_{n}\in\mathbb{Q}$ such that $x=\sum_{i=1}^{n}q_{i}l_{i},$ 
and either $n\ge2$, or $n=1$ with $q_{1}\ne1$. Choose real numbers $a_{1},...,a_{n}$ such that
\begin{equation} 
s:=\sum_{i=1}^{n}q_{i}a_{i} \end{equation} 
satisfies $s\notin\{0,a_{1},...,a_{n}\}$. This is possible because the forbidden conditions
\begin{equation} 
\sum_{i=1}^{n}q_{i}a_{i}=0, \quad \sum_{i=1}^{n}q_{i}a_{i}=a_{k} \ (k=1,...,n) 
\end{equation} 
define only finitely many proper affine hyperplanes in $\mathbb{R}^{n}$ so one can choose $(a_{1},...,a_{n})$ outside their union. Now define a $\mathbb{Q}$-linear functional $\varphi:Span_{\mathbb{Q}}(L)\rightarrow\mathbb{R}$ by prescribing
\begin{equation} 
\varphi(l)=0 \ (l\in L\setminus\{l_{1},...,l_{n}\}), \quad \varphi(l_{i})=a_{i} \ (i=1,...,n), \end{equation} 
and extending $\mathbb{Q}$-linearly. This is well defined because $L$ is a $\mathbb{Q}$-basis of $Span_{\mathbb{Q}}(L)$. Then $\varphi(L)\subseteq\{0,a_{1},...,a_{n}\}$, while $\varphi(x)=\sum_{i=1}^{n}q_{i}a_{i}=s\notin\varphi(L)$.

Extend to an additive map $H:\mathbb{R}\rightarrow\mathbb{R}$. Since $H(x)\notin H(L)$, choose an open interval $I$ around $H(x)$ disjoint from $H(L)$. Then $x\in H^{-1}(I)$ and $H^{-1}(I)\cap L=\emptyset$. So $x$ again has a $\tau_{H}$-open neighbourhood disjoint from $L$.
\end{itemize} 
In both cases, every point of $\mathbb{R}\setminus L$ has a $\tau_{\AD}$-open neighborhood disjoint from $L$. Hence $\mathbb{R}\setminus L$ is $\tau_{\AD}$-open, and therefore $L$ is $\tau_{\AD}$-closed.
\end{proof}

We now return to the proof of the theorem. By a theorem of Banakh and Protasov \cite{2}, under CH the set $\mathbb{R}\setminus\{0\}$ can be covered by countably many $\mathbb{Q}$-linearly independent sets. Thus there exist sets $L_{n}\subseteq\mathbb{R}\setminus\{0\}$, $n\in\mathbb{N}$, such that $\mathbb{R}\setminus\{0\}=\bigcup_{n\in\mathbb{N}}L_{n},$ 
and each $L_{n}$ is $\mathbb{Q}$-linearly independent. Let $A\subseteq\mathbb{R}$ be arbitrary. Then
\begin{equation} 
A=(A\cap\{0\})\cup\bigcup_{n\in\mathbb{N}}(A\cap L_{n}). \end{equation} 

The singleton $\{0\}$ is both $\tau_{\AD}$-closed and $\tau_{\Q}$-closed: Indeed, every Euclidean open set is open in the two topologies, hence $A\cap\{0\}$ is closed in both topologies, since the complement of this set is either $\R\setminus \{0\}$ or $\R$, which are both open.

For part~(a), Lemma~\ref{lem:7.6} implies that each $A\cap L_n$, being a subset of the $\mathbb{Q}$-linearly independent set $L_n$, is $\tau_{\Q}$-closed. Hence $A$ is an $F_\sigma$ set in $\tau_{\Q}$.

For part~(b), each $A\cap L_n$ is itself $\mathbb{Q}$-linearly independent, since linear independence is inherited by subsets. Lemma~\ref{lem:7.7} therefore implies that each $A\cap L_n$ is $\tau_{\AD}$-closed. Consequently, $A$ is an $F_\sigma$ set in $\tau_{\AD}$.

\end{proof}

\subsection{Particular Paradoxical Examples}
The following two examples, while not pathological per se, add to the list of paradoxical properties the two topologies have. First, the rationals are a closed, nowhere dense set:

\begin{theorem}
\label{thm:7.8}
The set of irrationals $\mathbb{R}\setminus\mathbb{Q}$ is $\tau_{\Q}$-open.
\end{theorem}

\begin{proof}
\begin{lemma} $\mathbb{R}\setminus\mathbb{Q}$ is $\tau_{\Q}$-open. 
\end{lemma} 
\begin{proof} 
Fix $a\in\mathbb{R}\setminus\mathbb{Q}$. We show that $(\mathbb{R}\setminus\mathbb{Q})-a$ is a $\tau_{\Q}$-neighborhood of $0$. Let $x\ne0$ be arbitrary. We must find $\epsilon>0$ such that for all rationals $\lambda$ with $0<\lambda<\epsilon$,
\begin{equation} 
\lambda x\in\mathbb{R}\setminus\mathbb{Q}-a\iff a+\lambda x\in\mathbb{R}\setminus\mathbb{Q}. 
\end{equation} 
\noindent \textbf{Case 1:} $x\in\mathbb{Q}\setminus\{0\}$. Then for every rational $\lambda$, $\lambda x\in\mathbb{Q}$, so $a+\lambda x$ is irrational because (irrational) + (rational) is irrational. Hence any $\epsilon>0$ works.

\noindent \textbf{Case 2:} $x\in\mathbb{R}\setminus\mathbb{Q}$. Consider the set of ``bad'' rationals $S:=\{\lambda\in\mathbb{Q}:a+\lambda x\in\mathbb{Q}\}.$ 
We claim that $S$ has at most one element. Indeed, if $\lambda_{1}\ne\lambda_{2}$ are in $S$, then both $a+\lambda_{1}x$ and $a+\lambda_{2}x$ are rational, so subtracting yields $(\lambda_{1}-\lambda_{2})x\in\mathbb{Q}.$ 
But $\lambda_{1}-\lambda_{2}\in\mathbb{Q}^{\times}$, hence $x\in\mathbb{Q}$ contradicting $x\in\mathbb{R}\setminus\mathbb{Q}$. Thus $S$ is either empty or a singleton $\{\lambda_{0}\}$.

If $S=\emptyset$, then $a+\lambda x$ is irrational for every rational $\lambda$, so any $\epsilon>0$ works. If $S=\{\lambda_{0}\}$, choose $\epsilon>0$ with $\epsilon<\lambda_{0}$ when $\lambda_{0}>0$ (and take any $\epsilon>0$ if $\lambda_{0}\le0$). Then no rational $\lambda$ with $0<\lambda<\epsilon$ equals $\lambda_{0}$, hence $a+\lambda x\notin\mathbb{Q}$ for all such $\lambda$.

This proves that $\mathbb{R}\setminus\mathbb{Q}$ is a $\tau_{\Q}$-neighborhood of $a$. Since $a\in\mathbb{R}\setminus\mathbb{Q}$ was arbitrary, $\mathbb{R}\setminus\mathbb{Q}$ is $\tau_{\Q}$-open.
\end{proof} 
\begin{lemma} 
$(\mathbb{R}\setminus\mathbb{Q})\cap I$ is comeager in $I$ for every nonempty Euclidean interval $I$. 
\end{lemma} 
\begin{proof} 

Indeed, the set $\mathbb{Q}\cap I$ is countable, hence meager in $I$ (with the subspace Euclidean topology). Therefore $(\mathbb{R}\setminus\mathbb{Q})\cap I=I\setminus(\mathbb{Q}\cap I)$ is comeager in $I$.
\end{proof} 
\begin{lemma} 
$(\mathbb{R}\setminus\mathbb{Q})\cap I\ne I$. 
\end{lemma} 
\begin{proof} 
Since $\mathbb{Q}$ is dense in $\mathbb{R}$, every nonempty interval $I$ contains a rational number. Hence $\mathbb{Q}\cap I\ne\emptyset$ so $(\mathbb{R}\setminus\mathbb{Q})\cap I\ne I$.
\end{proof} 
\end{proof}

\begin{theorem}
\label{thm:7.9}
Assume the Axiom of Choice. Then $\mathbb{R}\setminus\mathbb{Q}$ is also $\tau_{\AD}$-open in the topology $\tau_{\AD}$.
\end{theorem}

\begin{proof}
For each $x\in\mathbb{R}\setminus\mathbb{Q}$ one needs to display an open set in $\tau_{H}$ containing $x$. Indeed, choose an additive function $H_{x}$ with $H_{x}(1)=0$ and $H_{x}(x)=1$. This can be done by creating a Hamel Basis $\mathcal{B}$ containing $x$ and $1$. Such a basis is guaranteed to exist under the Axiom of Choice.

Then the basic $\tau_{\AD}$-open neighborhood $H_{x}^{-1}((1/2,3/2))$ contains $x$ and avoids $\mathbb{Q}$. Thus $\mathbb{Q}$ is $\tau_{\AD}$-closed and $\mathbb{R}\setminus\mathbb{Q}$ is $\tau_{\AD}$-open.
\end{proof}

\begin{remark}
\label{rem:7.10}
A similar result holds for the transcendental numbers.
\end{remark}

\begin{theorem}
\label{thm:7.11}
The rationals are nowhere dense in $\tau_{\Q}$.
\end{theorem}

\begin{proof}
Let $a\in \Q$. For every irrational $x$, $a+qx\not \in \Q$ for every $q\in \Q, q\neq 0$. So for every $\tau_{\Q}$ open set $U\ni a$, there exists $y\in U\cap (\R \setminus \Q)$. 

On the other hand, by Theorem~\ref{thm:7.8}, the closure of $\Q$ (in $\tau_{\Q}$) is precisely $\Q$. 

We infer that $int(\overline{\Q})=\emptyset$ (where both the closure and interior are taken with respect to $\tau_{\Q}$), i.e. $\Q$ is nowhere dense in $\tau_{\Q}$. 
\end{proof}

\begin{remark} 
$\Q$ is a perfect set in $\tau_{\Q}$. Indeed, as we saw, $\Q$ is $\tau_{Q}$-closed. If $a\in \Q$ and $U$ is a $\tau_{\Q}$-neighborhood of $a$, applying Definition~\ref{def:tauq} in the direction of 1 we can find a distinct rational point $a+q\in U$. 
\end{remark} 

\begin{theorem} Assuming AC,  
the rationals are nowhere dense in $\tau_{\AD}$. 
\end{theorem} 
\begin{proof} 



Choose a Hamel basis $\B$ containing $1$, and 
recall the definition of sets $S_{n,\B}$ from equation~(\ref{eq:sn}), in particular of $S_{1,\B}$. Since $1\in \B$, $\Q\subseteq S_{1,\B}$. Lemma~\ref{lem:7.4} shows that $S_{1,\B}$ is nowhere dense. On the other hand Theorem~\ref{thm:7.9} shows that $\Q$ is $\tau_{\AD}$-closed. So the interior of $\Q$ in $\tau_{\AD}$ is included in the interior of $S_{1,\B}$ in $\tau_{\AD}$, which is the empty set. 

Thus $\mathbb Q$ is closed and has empty interior, so it is nowhere dense in $(\mathbb R,\tau_{\AD})$.
\end{proof}

\section{Characterizing $\Q$-continuity notions on rational lines}
\label{sec:4} 

In Figure~\ref{fig-1} we have claimed a number of strict separations between the five classes. However, up to this point we have only proved (nonstrict) inclusions. 

In order to compare and distinguish these notions, it will be helpful to show that some of these weak continuity notions can be characterized by the specific nature of restrictions of the given function to all rational lines. 

We start with $\tau_{\Q}$-continuity. Although no longer a candidate, this concept has the following rather elegant characterization below: 

\begin{theorem}
For a function $f:\mathbb{R}\rightarrow\mathbb{R}$, the following are equivalent:
\begin{enumerate}
\item[(1)] $f$ is $\tau_{\Q}$-continuous.
\item[(2)] For every $x<y,$ the restriction $f|_{L(x,y)}$ is continuous in the usual subspace topology.
\item[(3)] For every $x<y$, the restriction $f|_{S(x,y)}$ is continuous in the usual subspace topology.
\end{enumerate}
\label{thm:qcont-char}
\end{theorem}

\begin{proof}
(1) $\Rightarrow$ (2). Fix $x<y$ put $d:=x-y\ne0$, and let $a\in L(x,y).$ Let $V\subseteq\mathbb{R}$
be an ordinary open neighborhood of $f(a)$, and set
$U:=f^{-1}(V)$. 
Since $f$ is $\tau_{\Q}$-continuous, $U$ is $\tau_{\Q}$-open at $a$, so $U-a$ is a $\tau_{\Q}$-neighborhood of $0$. By
Definition~\ref{def:tauq} applied to the directions $d$ and $-d$, there exists $\epsilon>0$ such that  $a+qd\in U \quad (q\in\mathbb{Q},|q|<\epsilon).$
But $L(x,y)=a+\mathbb{Q}d$, 
so the set $\{a+qd:q\in\mathbb{Q}, |q|<\epsilon\}$ is an ordinary neighborhood of $a$ inside
the subspace $L(x,y)$. Hence $f|_{L(x,y)}$ is continuous at $a$. Since $a$ was arbitrary,
$f|_{L(x,y)}$ is continuous on $L(x,y)$.

(2) $\Rightarrow$ (1). Fix $a\in\mathbb{R}$, and let $V\subseteq\mathbb{R}$ be an ordinary open neighborhood of
$f(a)$. Let $U:=f^{-1}(V)$. 
We show that $U$ is $\tau_{\Q}$-open at $a$. Let $h\ne0.$ Consider the rational affine line
$L(a,a+h)=a+\mathbb{Q}h.$ By assumption, $f|_{L(a,a+h)}$ is continuous at $a$ in the ordinary subspace topology.
Therefore there exists $\epsilon_{h}>0$ such that $f(a+qh)\in V$, therefore 
$a+qh\in U \quad \forall  q\in\mathbb{Q},|q|<\epsilon_{h}$. 
Since this holds for every $h\ne0$, Definition~\ref{def:tauq} implies that $U$ is $\tau_{\Q}$-open at $a$.
Hence $f$ is $\tau_{\Q}$-continuous.

(2) $\Rightarrow$ (3) is immediate because $S(x,y)\subseteq L(x,y).$

(3) $\Rightarrow$ (2). 
Assume that, for every $u<v$, the restriction
$
f\big|_{S(u,v)}
$
is continuous in the usual subspace topology. We prove directly that, for every
$x<y$, the restriction
$
f\big|_{L(x,y)}
$
is continuous in the usual subspace topology.

Fix $x<y$ and put $L:=L(x,y)$. Let $a\in L$, and let $V\subseteq\R$ be an
open neighborhood of $f(a)$. Choose points $u,v\in L$ such that
$
u<a<v.$
By continuity of $f|_{S(u,a)}$ at $a$, there exists $\delta_->0$ such that
\begin{equation} 
t\in S(u,a), |t-a|<\delta_-
\Longrightarrow\quad
f(t)\in V.
\end{equation} 
Likewise, by continuity of $f|_{S(a,v)}$ at $a$, there exists $\delta_+>0$
such that
\begin{equation} 
t\in S(a,v), |t-a|<\delta_+
 \Longrightarrow\quad
f(t)\in V.
\end{equation} 
Set
$
\delta:=\min\{\delta_-,\delta_+,a-u,v-a\}>0.$
We claim that
$
t\in L, |t-a|<\delta \Longrightarrow
f(t)\in V.$ 
Indeed, such a point $t$ belongs to $(u,v)$. If $t\le a$, then
$t\in S(u,a)$. To see this, write
$
u=x+q_u(y-x), a=x+q_a(y-x), t=x+q_t(y-x)$, 
with $q_u,q_a,q_t\in\Q$. Since $u\le t\le a$, we have
\begin{equation} 
\lambda:=\frac{q_t-q_u}{q_a-q_u}\in\Q\cap[0,1],
\end{equation} 
and, therefore, 
$t=(1-\lambda)u+\lambda a\in S(u,a).$
Hence $f(t)\in V$ by the choice of $\delta_-$. The case $t\ge a$ is
analogous: then $t\in S(a,v)$, and the choice of $\delta_+$ gives
$f(t)\in V$.

Thus there exists a neighborhood $
(a-\delta,a+\delta)\cap L$ 
of $a$ in the subspace $L$ whose image under $f$ is contained in $V$.
Therefore $f|_L$ is continuous at $a$. Since $a\in L$ was arbitrary,
$f|_L$ is continuous on $L$.

\end{proof}

\begin{theorem}
For a function $f:\mathbb{R}\rightarrow\mathbb{R}$, the following are equivalent:
\begin{enumerate}
\item[(1)] $f$ is bilaterally $\mathbb{Q}$-continuous.
\item[(2)] For every $x<y$ there exists a continuous function $G_{x,y}\in C(\mathbb{R})$ such that
\begin{equation} 
G_{x,y}(t)=f(t) \quad \forall t\in L(x,y) 
\end{equation} 
\end{enumerate}
Similarly, the following are equivalent:
\begin{enumerate}
\item[(3)] $f$ is $\mathbb{Q}$-continuous.
\item[(4)] There exists a family of functions $\{G_{x,y}\}_{x<y\in\mathbb{R}}$, $G_{x,y}:\mathbb{R}\rightarrow\mathbb{R}$ such that
\begin{itemize}
\item $G_{x,y}(t)=f(t) \quad \forall t\in L(x,y).$
\item The set of discontinuities of $G_{x,y}$ contains no points in $L(x,y)$.
\item The set of points $z$ for which some $G_{x,y}$ is discontinuous at $z$ has no accumulation points in $\R$, that is, it is discrete and closed. 
\end{itemize}
\end{enumerate}
\label{character-qcont}
\end{theorem}

\begin{proof}
(2) $\Rightarrow$ (1). Fix $x_{0}\in\mathbb{R}$. To prove bilateral-$\mathbb{Q}$-continuity at $x_{0}$, choose any
$\epsilon>0$. Whenever $x_{0}-\epsilon\le y<z\le x_{0}+\epsilon$ the function $G_{y,z}\in C(\mathbb{R})$ from (2) is a
witness on the interval $[y, z]$.

(1) $\Rightarrow$ (2). Fix $a<b$ and 
define
\begin{equation} 
G_{a,b}(t):=\lim_{z\rightarrow t,z\in L(a,b)}f(z). 
\label{eq:G_a,b}
\end{equation}
We have to prove several things: 
\begin{itemize} 
\item[(a). ] The definition is correct: the limit exists for every $t\in \R$. 
\item[(b). ] The function $G_{a,b}$ extends $f$ on $L(a,b)$. 
\item[(c). ] $G_{a,b}$ is continuous on $\R$.  
\end{itemize} 

Indeed: 
\begin{itemize} 
\item[(a). ] Let $t\in \R$ be arbitrary. By the bilateral $\Q$-continuity of $f$ on $\R$, there exists $\epsilon > 0$ such that, if we choose $t-\epsilon < x<t<y<t+\epsilon$, $x,y\in L(a,b)$ then there exists a continuous function $g_{x,y}:[x,y]\rightarrow \R$ such that $g=f$ on $L(a,b)\cap [x,y]$. 

We have: 
\begin{equation} 
\lim_{z\rightarrow t,z\in L(a,b)}f(z)=\lim_{z\rightarrow t,z\in L(a,b)}g_{x,y}(z)=g_{x,y}(t),
\label{eq:long} 
\end{equation} 
so the limit in~(\ref{eq:G_a,b}) is well defined. 

\item[(b). ] If $t\in L(a,b)$ then (by eq.~(\ref{eq:long})) $G_{a,b}(t)=g_{x,y}(t)=f(t)$ (since $g_{x,y}$ is equal to $f$ on $S(x,y)$, and continuous at $t$). 

\item[(c). ] Let $t\in \R$ be arbitrary. We have to prove that $G_{a,b}$ is continuous at $t$. Let $x,y$ be chosen as described in point (a). We prove that for every $w\in (x,y)$ 
\begin{equation} 
G_{a,b}(w)=g_{x,y}(w). 
\label{eq:ext} 
\end{equation} 
Proving equation~(\ref{eq:ext}) is enough to guarantee that $G_{a,b}$ is continuous at $t$, since $g_{x,y}$ is. 

Indeed, 
\begin{equation}
G_{a,b}(w)=\lim_{z\rightarrow w,z\in L(a,b)}f(z)= \lim_{z\rightarrow w,z\in L(a,b)}g_{x,y}(z)=g_{x,y}(w). 
\end{equation} 
The first equality is by~(\ref{eq:G_a,b}). The second follows since $g_{x,y}=f$ on $S(x,y)$. The third one follows from the continuity of $g_{x,y}$ at $w$. 

Since $t$ was chosen
arbitrarily, $G_{a,b}\in C(\mathbb{R})$.

\end{itemize}

(3) $\Rightarrow$ (4). Let $w$ be a point and $a<b$. Since $f$ is $\mathbb{Q}$-continuous at $w$, there
exists $\epsilon>0$ such that the function $f|_{L(a,b)\cap(w-\epsilon,w)}$ has a continuous extension
$g_{a,b}^{-}$ to $(w-\epsilon,w)$: consider the value $\epsilon^{\prime}$ given by the $\mathbb{Q}$-continuity of $f$ at $w$. Let
$0<\epsilon<\epsilon^{\prime}$ be so that $w-\epsilon\in L(a,b)$. Take an increasing sequence of points $w_{n}$ in $L(a,b)\cap(w-\epsilon,w)$ with $\lim_{n\rightarrow\infty}w_{n}=w$.
Consider $g_{n}^{-}$, the extension of $f$ on the interval $[w-\epsilon,w_{n}]$ guaranteed by $\mathbb{Q}$-continuity. Since $g_{n}^{-}$ is continuous, for all $u\in[w-\epsilon,w_{n}]$
\[ \lim_{\substack{z\in L(a,b)\cap[w-\epsilon,w_{n}]\\z\rightarrow u}}f(z)=\lim_{\substack{z\in L(a,b)\cap[w-\epsilon,w_{n}]\\z\rightarrow u}}g_{n}^{-}(z)=g_{n}^{-}(u). \]
Since the sequence $w_{n}$ is increasing, the last equation (in which the first term
does not depend on $n$) ensures that $g_{n+1}^{-}$ extends $g_{n}^{-}$ consistently on the common
domain $[w-\epsilon,w_{n}]$\footnote{We are using here the fact that two continuous functions that are equal on a dense rational grid are equal everywhere.}. So we can glue all functions $g_{n}^{-}$ to the function $g_{a,b}^{-}$ defined
on the whole interval $(w-\epsilon,w)$. Similarly we obtain a continuous extension $g_{a,b}^{+}$ on $(w,w+\epsilon^{\prime})$.

Let now $a<b$, let $L=L(a,b)$, and let $G_{a,b}$ be the function
defined by
\begin{equation} 
G_{a,b}(z)=
\begin{cases}
\alpha, &
\displaystyle
\lim_{\substack{x\to z\\x<z,\ x\in L}}f(x)
=
\lim_{\substack{x\to z\\x>z,\ x\in L}}f(x)
=\alpha,\\[1ex]
0, & \text{otherwise}.
\end{cases}
\label{G_a,b,2}
\end{equation} 


\begin{lemma}\label{lem:Gab-continuous-on-line}
If $w\in L$, then $G_{a,b}(w)=f(w)$ and $G_{a,b}$ is continuous
at $w$.
\end{lemma}

\begin{proof}
Fix $w\in L$. Since $f$ is $\mathbb Q$-continuous at $w$, there
exist $\varepsilon_-,\varepsilon_+>0$ such that the extension
condition in Definition~3.5 holds on every interval
$w-\varepsilon_-<y<z\leq w$ and on every interval
$w\leq y<z<w+\varepsilon_+.$
As in the preceding construction, these extensions can be glued to
give continuous functions $
g^-_{a,b}\colon (w-\varepsilon_-,w)\longrightarrow\mathbb R,
g^+_{a,b}\colon (w,w+\varepsilon_+)\longrightarrow\mathbb R,$
such that
\begin{equation} 
g^-_{a,b}(x)=f(x)
\quad
\bigl(x\in L\cap(w-\varepsilon_-,w)\bigr)
\end{equation} 
and
\begin{equation} 
g^+_{a,b}(x)=f(x)
\quad
\bigl(x\in L\cap(w,w+\varepsilon_+)\bigr).
\end{equation} 

We first observe that
\begin{equation}\label{eq:Gab-left-piece}
G_{a,b}(z)=g^-_{a,b}(z)
\qquad
(z\in(w-\varepsilon_-,w)).
\end{equation}
Indeed, fix such a $z$. Choose $\delta>0$ so small that $(z-\delta,z+\delta)\subset(w-\varepsilon_-,w).$
Since $g^-_{a,b}$ is continuous at $z$ and agrees with $f$ on
$L\cap(w-\varepsilon_-,w)$, we have
\begin{equation} 
\lim_{\substack{x\to z\\x<z,\ x\in L}} f(x)
 =
\lim_{\substack{x\to z\\x<z,\ x\in L}} g^-_{a,b}(x)
 =
g^-_{a,b}(z),
\end{equation} 
and likewise
\begin{equation} 
\lim_{\substack{x\to z\\x>z,\ x\in L}} f(x)
 =
\lim_{\substack{x\to z\\x>z,\ x\in L}} g^-_{a,b}(x)
 =
g^-_{a,b}(z).
\end{equation} 
Thus the first case in the definition of $G_{a,b}$ applies at $z$,
and~\eqref{eq:Gab-left-piece} follows.

Exactly the same argument gives
\begin{equation}\label{eq:Gab-right-piece}
G_{a,b}(z)=g^+_{a,b}(z)
\qquad
(z\in(w,w+\varepsilon_+)).
\end{equation}

It remains to determine the behavior of these two functions at $w$.
Choose $u\in L\cap(w-\varepsilon_-,w)$. 
By left $\mathbb Q$-continuity of $f$ at $w$, there exists a
continuous function $
h^-\colon [u,w]\longrightarrow\mathbb \R$
which agrees with $f$ on $S(u,w)$. Since $u,w\in L$, we have $S(u,w)=L\cap[u,w].$

Consequently, on the dense subset
$L\cap(u,w)\subset(u,w)$ 
the two continuous functions $h^-$ and $g^-_{a,b}$ both agree with
$f$. Hence
\begin{equation} 
h^-(x)=g^-_{a,b}(x)
\qquad (u<x<w).
\end{equation} 
It follows that
\begin{equation} 
\lim_{x\to w^-}g^-_{a,b}(x)
 =
\lim_{x\to w^-}h^-(x)
 =
h^-(w)
 =
f(w),
\end{equation} 
where the last equality holds because $w\in S(u,w)$.

Similarly, choosing $
v\in L\cap(w,w+\varepsilon_+)$ 
and using right $\mathbb Q$-continuity at $w$, we obtain
\begin{equation} 
\lim_{x\to w^+}g^+_{a,b}(x)=f(w).
\end{equation} 

Since $g^-_{a,b}=f$ on $L\cap(w-\varepsilon_-,w)$ and
$g^+_{a,b}=f$ on $L\cap(w,w+\varepsilon_+)$, the preceding two
equalities also give
\begin{equation} 
\lim_{\substack{x\to w\\x<w,\ x\in L}}f(x)
=
\lim_{\substack{x\to w\\x>w,\ x\in L}}f(x)
=
f(w).
\end{equation} 
Therefore the first case in the definition of $G_{a,b}$ applies at
$w$, and
\begin{equation} 
G_{a,b}(w)=f(w).
\end{equation} 

Finally, by~\eqref{eq:Gab-left-piece} and
\eqref{eq:Gab-right-piece},
\begin{equation} 
G_{a,b}(z)=
\begin{cases}
g^-_{a,b}(z), & w-\varepsilon_-<z<w,\\
f(w),         & z=w,\\
g^+_{a,b}(z), & w<z<w+\varepsilon_+.
\end{cases}
\end{equation} 
Since
\begin{equation} 
\lim_{z\to w^-}g^-_{a,b}(z)
=
f(w)
=
\lim_{z\to w^+}g^+_{a,b}(z),
\end{equation} 
we conclude that
\begin{equation} 
\lim_{z\to w}G_{a,b}(z)=G_{a,b}(w)=f(w).
\end{equation} 
Thus $G_{a,b}$ is continuous at $w$.
\end{proof}

Now we have to prove that the set $\mathcal{D}$,  defined as the set of points of discontinuities for some function $G_{x,y}(\cdot)$ constructed as described above, has no accumulation points in $\R$ (that is it is closed and discrete). This follows in the
following way: 

Assume that $w\in \R$. Since $f$ is $\Q$-continuous at $w$, there exist intervals $(w-\varepsilon,w)$ 
such that for every $w-\varepsilon < y<z\leq w$ there is a continuous function $g_{y,z}$ that coincides with $f$ on $S(y,z)$.  

The following result (and its upper analog) shows that 
$\mathcal{D}$ has no accumulation points: 

\begin{lemma}\label{lem:D-no-left-accumulation}
We have
\begin{equation} 
D\cap (w-\varepsilon,w)=\varnothing.
\end{equation} 
\end{lemma}

\begin{proof}
Suppose, toward a contradiction, that
$t\in D\cap(w-\varepsilon,w).$
Then $G_{a,b}$ is discontinuous at $t$ for some $a<b$. Since the rational line $L(a,b)$ is dense in $\mathbb R$, choose $
a_1,a_2\in L(a,b)\cap(w-\varepsilon,w)$ 
such that $a_1<t<a_2.$ 
By the choice of $\varepsilon$ and the $\mathbb Q$-continuity of
$f$ at $w$, there exists a continuous function
$h\colon [a_1,a_2]\longrightarrow\mathbb R$ 
which agrees with $f$ on $S(a_1,a_2)$.
Hence
\begin{equation} 
h(x)=f(x)
\qquad
\text{for every }x\in L(a,b)\cap[a_1,a_2].
\end{equation} 

We claim that $
G_{a,b}(u)=h(u)
\text{ for every }u\in(a_1,a_2).$
Indeed, fix such a point $u$. Since $h$ is continuous at $u$ and agrees with
$f$ on $L(a,b)\cap[a_1,a_2]$, we have
\begin{equation} 
\lim_{\substack{x\to u\\x<u,\ x\in L(a,b)}} f(x)
=
\lim_{\substack{x\to u\\x<u,\ x\in L(a,b)}} h(x)
=
h(u),
\end{equation} 
and similarly
\begin{equation} 
\lim_{\substack{x\to u\\x>u,\ x\in L(a,b)}} f(x)
=
\lim_{\substack{x\to u\\x>u,\ x\in L(a,b)}} h(x)
=
h(u).
\end{equation} 
Therefore the first case in the definition of $G_{a,b}$ applies at
$u$, and consequently $G_{a,b}(u)=h(u).$ Thus $
G_{a,b}|_{(a_1,a_2)}=h|_{(a_1,a_2)}.$
Since $h$ is continuous and $t\in(a_1,a_2)$, it follows that
$G_{a,b}$ is continuous at $t$, contradicting the choice of $t\in 
D\cap(w-\varepsilon,w)$. 
\end{proof}

(4) $\Rightarrow$ (3). Let $\mathcal{D}$ be the set of points $z\in \R$ where some function $G_{x,y}$ is discontinuous at $z$. 

\begin{lemma} 
If $w\notin\mathcal{D}$ then $f$ is bilaterally $\mathbb{Q}$-continuous (hence $\Q$-continuous) at $w$. 
\end{lemma}  
\begin{proof}
Choose an open interval around
$w$, say $(w-\varepsilon,w+\varepsilon)$, disjoint from $\mathcal{D}$. Let $w-\epsilon<y<z<w+\varepsilon$. The fact that $f$ can
be extended continuously on $[y, z]$ is witnessed by the function $G_{y,z}$ restricted
to the given interval. This function is continuous, since $[y,z]\cap\mathcal{D}$ is empty. 
\end{proof} 

If, on the other hand, $w\in\mathcal{D},$ then, since $\mathcal{D}$ is isolated, one can choose an
interval $(w-\epsilon,w+\epsilon)$ such that $\mathcal{D}\cap(w-\epsilon,w+\epsilon)=\{w\}$. The following result proves that $f$ is $\Q$-continuous at such a $w$: 

\begin{lemma} 
Function $f$ satisfies the conditions in Definition~\ref{def:qcont} for $\Q$-continuity from the left (right) at $w$ on intervals $(w-\varepsilon,w)$ and $(w,w+\varepsilon)$, respectively. 
\end{lemma}
\begin{proof} 
We only prove the $\Q$-continuity from the left; proving $\Q$-continuity from the right is completely analogous. 

Indeed, let
$w-\epsilon<y<z\leq w$. Each function $G_{y,z}$ is continuous on $[y, z]$, and 
extends $f$ on $S(y,z)$. 

The remaining case, $w-\epsilon<y<z=w$, is easy as well:  consider instead function $G_{y,w}$. Since $w\in L(y,w)$ and the discontinuity set of $G_{y,w}$ does not contain any point in that set, $G_{y,w}$ is continuous at $w$. Since $D\cap (w-\varepsilon,w)=\emptyset$, $G_{y,w}$ is continuous on $[y,w]$. Hence the function $G_{y,w}$ satisfies the conditions required from $\Q$-continuity from the left at $w$.  
\end{proof} 
\end{proof}

\begin{remark} 
There is a useful parallel with the notion of a \emph{linearly continuous function}:
a function on a real vector space is called linearly continuous when its restriction
to every affine real line is continuous; see Banakh and Maslyuchenko
\cite{45}.  Theorem~\ref{thm:qcont-char} may be regarded as a
rational-scalar counterpart of this notion: $\tau_{\mathbb Q}$-continuity is
equivalent to ordinary subspace continuity on every affine $\Q$-line.
(Bilateral) $\Q$-continuity impose the stronger requirements that each such
rational trace extends continuously to the corresponding real line.
\end{remark}

\section{Constructions and Separations}
\label{sec:5} 

We start with an auxiliary result that will be useful to construct functions by gluing them along rational lines: 

\begin{theorem}
\label{well-ordering-lines}

Assume the Axiom of Choice. The family of one-dimensional affine \(\Q\)-subspaces of \(\R\) (rational lines), which has cardinality $\mathfrak{c}$, admits a well-ordering \((L_\gamma)_{\gamma<\mathfrak{c}}\) such that
\begin{equation} 
 L_\gamma\cap (\bigcup_{\beta<\gamma}L_\beta)
\end{equation} 
is finite for every \(\gamma<\mathfrak{c}\).
\end{theorem}

\begin{proof}
We prove the corresponding assertion for every \(\Q\)-vector space \(V\), by induction on \(\kappa=\dim_{\Q}V\).

If \(\kappa\le\aleph_0\), then \(V\) is countable, because \(\Q\) is countable and every vector has finite support. Hence the family of affine \(\Q\)-lines is countable. Enumerate it as \(L_0,L_1,\ldots\). Distinct affine lines meet in at most one point, so
\begin{equation} 
 \left|L_n\cap (\bigcup_{m<n}L_m)\right|\le n.
\end{equation} 

Now let \(\kappa>\aleph_0\), and assume the result for all smaller dimensions. Choose a basis \(B=\{b_\xi:\xi<\kappa\}\), where \(\kappa\) is identified with its initial ordinal, and put
\begin{equation} 
 V_\alpha=\Span\{b_\xi:\xi<\alpha\}\qquad(\alpha<\kappa).
\end{equation} 
Every affine line \(L\) is determined by two vectors of finite support. Consequently there is a unique \(\rho(L)<\kappa\) such that
\begin{equation} 
 L\subseteq V_{\rho(L)+1}
 \quad\text{and}\quad
 L\nsubseteq V_{\rho(L)}.
\end{equation} 
For \(\alpha<\kappa\), let
\begin{equation} 
 \mathcal L_\alpha=\{L:\rho(L)=\alpha\}.
\end{equation} 
Since \(\dim_{\Q}V_{\alpha+1}=|\alpha+1|<\kappa\), the induction hypothesis gives a good well-order of all affine lines in \(V_{\alpha+1}\). Restrict that order to \(\mathcal L_\alpha\).

Finally, order the whole family lexicographically: first by increasing \(\alpha\), then by the chosen order inside \(\mathcal L_\alpha\). Fix \(L\in\mathcal L_\alpha\). Every line in an earlier layer is contained in \(V_\alpha\). Moreover, $|L\cap V_\alpha|\le1$, 
because two points of \(L\cap V_\alpha\) would force the entire affine line \(L\) to lie in \(V_\alpha\), contrary to the definition of \(\rho(L)\). Thus all earlier layers contribute at most one point. Earlier lines in the same layer contribute only finitely many points by the restricted induction order. Their union therefore meets \(L\) in a finite set, as required.
\end{proof}

\begin{theorem}
Let $A\subseteq\mathbb{Q}$ be finite, and let $f:A\rightarrow\mathbb{R}$ be arbitrary. Then one can
construct an extension $g:\mathbb{Q}\rightarrow\mathbb{R}$ such that:
\begin{enumerate}
\item[(1)] $g$ is continuous on $\mathbb{Q}$ (with the subspace topology inherited from $\mathbb{R}$),
\item[(2)] $g|_A = f$,
\item[(3)] there is no function $G:\mathbb{R}\rightarrow\mathbb{R}$ with $G|_{\mathbb{Q}}=g$ that is continuous on some
nondegenerate interval.
\end{enumerate}
\label{thm:cons2} 
\end{theorem}

\begin{proof}
Fix a countable dense set of irrationals. One explicit choice is as
follows: enumerate $\mathbb{Q}$ as $(r_k)_{k\ge 1}$, and set $\alpha_{k,m} = r_k + \frac{\sqrt{2}}{m} \quad (k,m\in\mathbb{N}, m\ge 1).$
Then each $\alpha_{k,m}\notin\mathbb{Q}$, and the set $D = \{\alpha_{k,m} : k,m\in\mathbb{N}, m\ge 1\} \subseteq \mathbb{R}\setminus\mathbb{Q}$
is dense in $\mathbb{R}$, because for any nonempty open interval $I$ one can pick a rational
$r_k\in I$ and then choose $m$ large enough so that $r_k+\sqrt{2}/m\in I$.

The following is a key observation: To prevent the existence of an interval on which some extension is continuous, it suffices to ensure that \empty{every} nondegenerate interval contains a point $x$
at which the rational limit of $g$ fails to exist. 

\begin{lemma} 
Let $g:\mathbb{Q}\rightarrow\mathbb{R}$ and let $F:\mathbb{R}\rightarrow\mathbb{R}$ satisfy $F|_{\mathbb{Q}}=g$. If $F$ is
continuous at some $x\in\mathbb{R}$, then
$ \lim_{\substack{q\rightarrow x\\ q\in\mathbb{Q}}} g(q)$
exists in $\mathbb{R}$ and equals $F(x)$. In particular, if the above rational limit does not exist
at $x$, then no extension $F$ can be continuous at $x$.
\label{lemma:extension}
\end{lemma}

\begin{proof}
Assume $F$ is continuous at $x$. Let $(q_n)\subseteq\mathbb{Q}$ with $q_n\rightarrow x$. Then $g(q_n)=F(q_n)\rightarrow F(x)$ by continuity of $F$ at $x$. Hence every rational sequence converging
to $x$ has image converging to the same limit $F(x)$, i.e. the rational limit exists
and equals $F(x)$.
\end{proof}

Now let $\emptyset \neq A = \{a_1 < \dots < a_n\} \subseteq \mathbb{Q}$ be finite and let $f:A\rightarrow\mathbb{R}$ be arbitrary. We construct the desired function $g$ in several steps.  
First, define $g_0:\mathbb{R}\rightarrow\mathbb{R}$ by
piecewise linear interpolation:
\begin{itemize}
\item $g_0(x) = f(a_1)$ for $x\le a_1$;
\item for each $i=1,\dots,n-1$ and $x\in[a_i, a_{i+1}]$,
\begin{equation} g_0(x) = f(a_i) + \frac{x-a_i}{a_{i+1}-a_i} (f(a_{i+1})-f(a_i)); 
\end{equation} 
\item $g_0(x) = f(a_n)$ for $x\ge a_n$.
\end{itemize}
By construction $g_0$ is continuous on $\mathbb{R}$ and satisfies $g_0(a_i) = f(a_i)$ for all $i$.


Next, enumerate the dense irrational set $D$ as $(\beta_{j})_{j\ge1}$ (any enumeration will do). For $j\ge1$ define
\begin{equation} 
\varphi_{j}(x)=\sin\left(\frac{1}{x-\beta_{j}}\right), \quad x\in\mathbb{R}\setminus\{\beta_{j}\}. \end{equation} 

Since $\beta_{j}\notin\mathbb{Q}$ the restriction $\varphi_{j}|_{\mathbb{Q}}$ is a well-defined continuous function on $\mathbb{Q}$ (it is the restriction of a continuous function on $\mathbb{R}\setminus\{\beta_{j}\}$, and no point of $\mathbb{Q}$ is removed). Define, for $q\in\mathbb{Q}$,
\begin{equation} 
H(q)=\sum_{j=1}^{\infty}2^{-j}\varphi_{j}(q). 
\end{equation} 
Since $|\varphi_{j}(q)|\le1$ for all $q\in\mathbb{Q}$, the Weierstrass M-test implies that the series converges uniformly on $\mathbb{Q}$, hence $H:\mathbb{Q}\rightarrow\mathbb{R}$ is continuous.

\vspace{0.2cm}

Now, let $P(x)=\prod_{a\in A}(x-a).$ $P$ is a polynomial with rational coefficients, so $P|_{\mathbb{Q}}$ is continuous on $\mathbb{Q}$ and $P(a)=0$ for all $a\in A$. 

Finally define $g:\mathbb{Q}\rightarrow\mathbb{R}$ by
\begin{equation} g(q)=g_{0}(q)+P(q)H(q), \quad q\in\mathbb{Q}. 
\label{9.4} 
\end{equation} 
We can complete equation~(\ref{9.4}) to the case $A=\emptyset$ by taking in this case
\begin{equation} 
g_0\equiv 0\mbox{ and }P\equiv 1. 
\end{equation} 

\begin{lemma}
\label{thm:5.5}
The function $g:\mathbb{Q}\rightarrow\mathbb{R}$ constructed above satisfies:
\begin{enumerate}
    \item[(1)] $g$ is continuous on $\mathbb{Q}$,
    \item[(2)] $g|_{A}=f$,
    \item[(3)] no extension $F:\mathbb{R}\rightarrow\mathbb{R}$ with $F|_{\mathbb{Q}}=g$ is continuous on any nondegenerate interval.
\end{enumerate}
\end{lemma}

\begin{proof}
(1) The restriction $g_{0}|_{\mathbb{Q}}$ is continuous on $\mathbb{Q}$ because $g_{0}$ is continuous on $\mathbb{R}$. Also, $P|_{\mathbb{Q}}$ is continuous on $\mathbb{Q}$ and $H$ is continuous on $\mathbb{Q}$ by uniform convergence. Therefore the function $q\mapsto P(q)H(q)$ is continuous on $\mathbb{Q}$, and so is $g$.

(2) If $a\in A$ then $P(a)=0$, hence $g(a)=f(a).$

(3) Let $I\subseteq\mathbb{R}$ be a nondegenerate open interval. Since $D$ is dense, pick $j$ such that $\beta_{j}\in I$. We claim that the rational limit $\lim_{q\rightarrow\beta_{j},q\in\mathbb{Q}}g(q)$ does not exist. Write $H(q)=2^{-j}\varphi_{j}(q)+\sum_{m\ne j}2^{-m}\varphi_{m}(q).$
Define
\begin{equation} 
R(q):=\sum_{m\ne j}2^{-m}\varphi_{m}(q) \quad (q\in\mathbb{Q}), \end{equation}
and note that for each $m\ne j$ the value $\varphi_{m}(\beta_{j})$ is well-defined because $\beta_{j}\ne\beta_{m}$. Set $L:=\sum_{m\ne j}2^{-m}\varphi_{m}(\beta_{j}),$ 
which converges absolutely since $|\varphi_{m}(\beta_{j})|\le1$.
\begin{lemma}  We have:
\begin{equation} 
\lim_{q\rightarrow\beta_{j}}R(q)=L.
\end{equation} 
\end{lemma} 
\begin{proof} 
Indeed, fix $\epsilon>0$ and choose $N\geq j$ so that $\sum_{m>N}2^{-m}<\epsilon/4$. By continuity of each $\varphi_{m}$ at $\beta_{j}$ for $m\in\{1,...,N\}\setminus\{j\}$, there exists $\delta>0$ such that whenever $q\in\mathbb{Q}$ and $|q-\beta_{j}|<\delta$:
\begin{equation} 
\sum_{\stackrel{1\le m\le N}{m\neq j}}2^{-m}|\varphi_{m}(q)-\varphi_{m}(\beta_{j})|<\epsilon/2. 
\end{equation} 
For the tail we use $|\varphi_{m}|\le1$ to get
\begin{equation} 
\sum_{m>N}2^{-m}|\varphi_{m}(q)-\varphi_{m}(\beta_{j})|\le2\sum_{m>N}2^{-m}<\epsilon/2.
\end{equation}  

Combining these inequalities yields $|R(q)-L|<\epsilon$ for all $q\in\mathbb{Q}$ with $|q-\beta_{j}|<\delta$, proving the lemma.
\end{proof} 

Crucially, $P(\beta_{j})\ne0$ because $P$ has only rational roots (namely the elements of $A$), whereas $\beta_{j}$ is irrational.

Now choose two rational sequences $(q_{n})$ and $(q_{n}^{\prime})$ with $q_{n}\rightarrow\beta_{j}$ and $q_{n}^{\prime}\rightarrow\beta_{j}$ such that
\begin{equation} 
\varphi_{j}(q_{n})=\sin\left(\frac{1}{q_{n}-\beta_{j}}\right)\rightarrow1, \mbox{ and }
\end{equation} 
\begin{equation} 
\varphi_{j}(q_{n}^{\prime})=\sin\left(\frac{1}{q_{n}^{\prime}-\beta_{j}}\right)\rightarrow-1. 
\end{equation} 
Such sequences exist because $\sin(1/t)$ oscillates between $-1$ and $1$ arbitrarily close to $t=0$, and rationals are dense.

We have 
\begin{equation} 
g(q_{n})=g_{0}(q_{n})+P(q_{n})\left(2^{-j}\varphi_{j}(q_{n})+\sum_{m\ne j}2^{-m}\varphi_{m}(q_{n})\right)
\end{equation} 
hence 
\begin{equation} 
\lim_{n\rightarrow \infty} g(q_n)= g_{0}(\beta_{j})+P(\beta_{j})(2^{-j}\cdot1+L)
\end{equation} 
Similarly, it holds that 
\begin{equation} 
\lim_{n\rightarrow \infty} g(q_{n}^{\prime})= g_{0}(\beta_{j})+P(\beta_{j})(2^{-j}\cdot(-1)+L). 
\end{equation} 
These two limits differ by $2\cdot2^{-j}P(\beta_{j})\ne0$. Hence the rational limit of $g$ at $\beta_{j}$ does not exist.

By Lemma~\ref{lemma:extension}, no extension $F$ of $g$ to $\mathbb{R}$ can be continuous at $\beta_{j}$. Since $\beta_{j}\in I$, it follows that no such $F$ can be continuous on the interval $I$.
\end{proof}
\end{proof} 

We can use Theorems~\ref{well-ordering-lines} and ~\ref{thm:cons2} to strongly separate $\tau_{\Q}$-continuity and 
$\mathbb Q$-continuity: 

\begin{theorem}\label{thm:tauQ-nowhere-Q}
Assume the Axiom of Choice. There exists a function
$f\colon\mathbb R\to\mathbb R$ which is $\tau_{\Q}$-continuous but is
$\mathbb Q$-continuous at no point.
\end{theorem}

\begin{proof}
Let $(L_\alpha)_{\alpha<\mathfrak{c}}$ be the well-ordering of the rational
affine lines supplied by Theorem~9.1, so that
$F_\alpha:=
L_\alpha\cap\bigcup_{\beta<\alpha}L_\beta$ 
is finite for every $\alpha<\mathfrak{c}$.

For every $\alpha<\mathfrak{c}$, fix an affine parametrization $
\varphi_\alpha\colon\mathbb Q\longrightarrow L_\alpha, 
\varphi_\alpha(q)=a_\alpha+q d_\alpha $,
where $d_\alpha\neq0$.

We construct $f$ by transfinite recursion. Suppose that $f$ has already
been defined on $
\bigcup_{\beta<\alpha}L_\beta.$ 
Put $
A_\alpha:=\varphi_\alpha^{-1}(F_\alpha)\subseteq\mathbb Q.$ 
Since $F_\alpha$ is finite, so is $A_\alpha$. Define
$
u_\alpha\colon A_\alpha\to\mathbb R,$ 
$u_\alpha(q):=f(\varphi_\alpha(q)).$
By Theorem~\ref{thm:cons2}, there exists a continuous function $
g_\alpha\colon\mathbb Q\to\mathbb R$ 
such that $
g_\alpha|_{A_\alpha}=u_\alpha$ 
and such that no extension $
G\colon\mathbb R\to\mathbb R, G|_{\mathbb Q}=g_\alpha,$
is continuous on any nondegenerate interval.

We now define $f$ on the new points of $L_\alpha$ by
\begin{equation} 
f(\varphi_\alpha(q)):=g_\alpha(q)
\end{equation} 
whenever $
\varphi_\alpha(q)\notin
\bigcup_{\beta<\alpha}L_\beta.$ 
On $F_\alpha$ this agrees with the values already assigned, by the
choice of $g_\alpha$. Thus after completing the recursion, $f$ is
well defined on all of $\mathbb R$.

For every $\alpha$ we have
\begin{equation} 
f(\varphi_\alpha(q))=g_\alpha(q)
\qquad(q\in\mathbb Q).
\end{equation} 
Since $g_\alpha$ is continuous on $\mathbb Q$, it follows that
$f|_{L_\alpha}$ is continuous in the usual subspace topology.
Therefore, by Proposition~\ref{thm:qcont-char}, $f$ is $\tau_{\Q}$-continuous.

We show that $f$ is not $\mathbb Q$-continuous at any point.
Fix $x\in\mathbb R$ and $\varepsilon>0$. Choose
$x<y<z<x+\varepsilon$ 
and let $L_\alpha=L(y,z)$. We claim that
$f|_{S(y,z)}$ admits no continuous extension to $[y,z]$.

Indeed, write $y=\varphi_\alpha(r),
z=\varphi_\alpha(s)$ 
for distinct $r,s\in\mathbb Q$. After interchanging $r$ and $s$ if
necessary, assume $r<s$. Then
\begin{equation} 
S(y,z)=\{\varphi_\alpha(q):q\in\mathbb Q\cap[r,s]\}.
\end{equation} 

Suppose, toward a contradiction, that there were a continuous function
$h\colon[y,z]\to\mathbb R$
such that $h(t)=f(t), (t\in S(y,z)).$
Define
$H(t):=h(\varphi_\alpha(t)),
\qquad t\in[r,s].$ 
Then $H$ is continuous on $[r,s]$, and for every
$q\in\mathbb Q\cap[r,s]$, $
H(q)
=
h(\varphi_\alpha(q))
=
f(\varphi_\alpha(q))
=
g_\alpha(q)$. 

We can therefore define a function
$\widetilde H\colon\mathbb R\to\mathbb R$ by
\begin{equation} 
\widetilde H(t)=
\begin{cases}
H(t), & t\in[r,s],\\
g_\alpha(t), & t\in\mathbb Q\setminus[r,s],\\
0, & t\in(\mathbb R\setminus\mathbb Q)\setminus[r,s].
\end{cases}
\end{equation} 
This is well defined and satisfies
$\widetilde H|_{\mathbb Q}=g_\alpha.$
Moreover, $\widetilde H$ is continuous on the nondegenerate open
interval $(r,s)$, since there it coincides with $H$. This contradicts
the defining property of $g_\alpha$ supplied by Theorem~9.2.

Hence $f|_{S(y,z)}$ has no continuous extension to $[y,z]$. Since such
$y,z$ can be chosen inside every right neighborhood of $x$, $f$ is
not right-$\mathbb Q$-continuous at $x$.

As $x$ was arbitrary, $f$ is $\mathbb Q$-continuous at no point.
\end{proof}

A similarly strong separation holds between $\tau_{\AD}$-continuity and $\Q$-continuity: 

\begin{theorem}
\label{thm:5.10}
Assuming a Hamel basis exists, then there exist $\tau_{\AD}$-continuous functions that are $\mathbb{Q}$-continuous at no point. 
\end{theorem}
\begin{proof} Take $H$ to be a discontinuous, nonsurjective, additive function. For instance, define a nonzero function $H$ on a Hamel basis $\B$ so that  $H(\B)$ does not span $\R$ over $\Q$. 

If $\alpha\notin H(\mathbb{R})$ then the set $A:=H^{-1}((-\infty,\alpha))$ is $\tau_{H}$-clopen. Indeed, $A$ is $\tau_H$-open by definition. On the other hand $\R\setminus A=H^{-1}((\alpha,\infty))$, so $A$ is $\tau_H$-closed as well. 

Because of this, the indicator function $\mathbbm{1}_{A}$ is $\tau_{H}$-continuous: indeed, for every open set $V$, $\mathbbm{1}_{A}^{-1}(V)$ is either $\emptyset, A,\R\setminus A$, or $\R$, depending on whether $0,1\in V$ or not.

On the other hand we show that $\mathbbm{1}_A$ is not $\Q$-continuous at any point of $\mathbb{R}$.
Indeed, fix $x\in\mathbb{R}$ and $\varepsilon>0$. We shall prove that the
extension condition for $\Q$-continuity fails on an interval contained in
$(x,x+\varepsilon)$.

Since $H$ is discontinuous and additive, 
by a well-known result (see e.g. \cite{12}) 
its graph is dense in
$\mathbb{R}^2$. Hence there exist
$x<y<z<x+\varepsilon$ such that $H(y)<\alpha<H(z).$
For instance, by density of the graph of $H$, we can take $y\in (x,x+\varepsilon/3), z\in (x+2\varepsilon/3,x+\varepsilon)$.

For every $t\in\mathbb{Q}\cap[0,1]$, additivity and $\mathbb{Q}$-linearity of $H$ imply
\begin{equation} 
H((1-t)y+tz)=(1-t)H(y)+tH(z).
\end{equation} 

Let
\begin{equation} 
t_0=\frac{\alpha-H(y)}{H(z)-H(y)}.
\end{equation} 

Then $0<t_0<1$ and $(1-t_0)H(y)+t_0H(z)=\alpha.$
Moreover, $t_0\notin\mathbb{Q}$, since otherwise
$H((1-t_0)y+t_0z)=\alpha,$ contradicting the assumption that $\alpha\notin H(\mathbb{R})$.

Choose rational sequences $(r_n)$ and $(s_n)$ satisfying $r_n<t_0<s_n,
r_n\to t_0,
s_n\to t_0.$  Define $
u_n=(1-r_n)y+r_nz,
v_n=(1-s_n)y+s_nz$, and
$u_0=(1-t_0)y+t_0z.$

Then $u_n,v_n\in S(y,z)$ and
$u_n\to u_0,
v_n\to u_0.$ Since the affine function
$
t\mapsto (1-t)H(y)+tH(z)$ is strictly increasing, we obtain $H(u_n)<\alpha<H(v_n)$ for every $n$. Therefore $\mathbbm{1}_A(u_n)=1,
\mathbbm{1}_A(v_n)=0$ for all $n$.

Suppose that $\mathbbm{1}_A(\cdot) |_{S(y,z)}$ admitted a continuous extension $g:[y,z]\to\mathbb{R}$. Then $g(u_n)=1,
g(v_n)=0$ 
for every $n$. Since both sequences converge to $u_0$, continuity of $g$ at $u_0$ yields $
g(u_n)\to g(u_0)
\quad\text{and}\quad
g(v_n)\to g(u_0)$, which is impossible.

Hence $\mathbbm{1}_A|_{S(y,z)}$ has no continuous extension to $[y,z]$.

Since such an interval may be chosen inside every right neighborhood of $x$, the function $\mathbbm{1}_A$ is not right-$\Q$-continuous at $x$.

Because $x$ was arbitrary, $\mathbbm{1}_A$ is nowhere $\Q$-continuous.
\end{proof}

Next we compare bilateral $\mathbb{Q}$-continuity and  $\tau_{\AD}$-continuity:

\begin{theorem}
\label{thm:5.7}
Assume that a Hamel basis exists.  
There exists a bilaterally $\mathbb{Q}$-continuous function $f:\mathbb{R}\rightarrow\mathbb{R}$ that is not continuous at $0$ with respect to $\tau_{\AD}$. In particular, (bilateral) $\mathbb{Q}$-continuity does not imply $\tau_{\AD}$-continuity.
\end{theorem}

\begin{proof}
Fix a Hamel basis $\B$ of $\mathbb{R}$ over $\mathbb{Q}$. Every $x\in\mathbb{R}$ admits a unique finite expansion
\begin{equation} 
x=\sum_{j=1}^{m}q_{j}b_{j}, \quad q_{j}\in\mathbb{Q}, \ b_{j}\in \B \text{ distinct.} 
\label{ham:exp}
\end{equation} 
Define $f:\mathbb{R}\rightarrow\mathbb{R}$ by
\begin{equation}\label{*}
\begin{cases}
f(0):=0 \\
f(x):=\max\{|\pi_{\B,b}(x)|:b\in Supp_{\B}(x)\} \quad (x\ne0)
\end{cases}
\end{equation}

\begin{lemma}
\label{lem:5.8}
The function $f$ defined by (\ref{*}) is bilaterally $\mathbb{Q}$-continuous. 
\end{lemma}

\begin{proof}[Proof of Lemma \ref{lem:5.8}]
Fix $y\ne z$ and write $d:=y-z\ne0$. Then $L(y,z)=\{z+qd:q\in\mathbb{Q}\}$. Let $S:=Supp_{\B}(z)\cup Supp_{\B}(d)$, which is finite. For each $b\in S$ set $a_{b}:=\pi_{\B,b}(z)\in\mathbb{Q}$ and $c_{b}:=\pi_{B,b}(d)\in\mathbb{Q}$. For any $t\in\mathbb{Q}$ the Hamel coefficients of $z+td$ satisfy $\pi_{\B,b}(z+td)=a_{b}+tc_{b} \quad (b\in S)$, 
and $\pi_{\B,b}(z+td)=0$ for $b\notin S$. Therefore, for rational $q\in\mathbb{Q}$,
\begin{equation} 
f(z+qd)=\max_{b\in S}|a_{b}+qc_{b}|. \end{equation} 
Define $h_{y,z}:\mathbb{R}\rightarrow\mathbb{R}$ by
\begin{equation} 
h_{y,z}(t):=\max_{b\in S}|a_{b}+tc_{b}|. 
\end{equation} 

This is a maximum of finitely many continuous (indeed affine) functions, hence is continuous. Now define $g_{y,z}:\mathbb{R}\rightarrow\mathbb{R}$ by
\begin{equation} 
g_{y,z}(x):=h_{y,z}\left(\frac{x-z}{y-z}\right). 
\end{equation} 
Since $x\mapsto(x-z)/(y-z)$ is continuous and $h_{y,z}$ is continuous, $g_{y,z}$ is continuous. If $x\in L(y,z)$, then $x=z+q(y-z)=z+qd$ for some $q\in\mathbb{Q}$, hence
\begin{equation} 
g_{y,z}(x)=h_{y,z}(q)=f(z+qd)=f(x). 
\end{equation} 
Thus $g_{y,z}$ is a continuous extension of $f|_{L(y,z)}$, proving bilateral $\mathbb{Q}$-continuity.
\end{proof}

We next show that $f$ is not $\tau_{\AD}$-continuous at $0$. Indeed, let $U$ be an arbitrary basic $\tau_{\AD}$-neighborhood of $0$, so
\begin{equation} 
U=\bigcap_{i=1}^{n}h_{i}^{-1}((-\delta_{i},\delta_{i})) 
\end{equation} 
for some $h_{1},...,h_{n}\in\mathcal{H}$ and $\delta_{1},...,\delta_{n}>0$. Every neighborhood contains an open neighborhood of the particular symmetric form (If $n=0$, then $U=\mathbb{R}$ and we may pick any $b_{0}\in B$ and take $x:=2b_{0}$, so $x\in U$ and $f(x)=2>1$. Thus we may assume $n\ge1$.)

Define the $\mathbb{Q}$-linear map $H:\mathbb{R}\rightarrow\mathbb{R}^{n}$ by $H(x):=(h_{1}(x),...,h_{n}(x)). $
Choose $n+1$ distinct basis elements $b_{0},...,b_{n}\in B$ and set $v_{j}:=H(b_{j})\in\mathbb{R}^{n}$ for $0\le j\le n$. Since $v_{0},...,v_{n}$ are $n+1$ vectors in the $n$-dimensional real vector space $\mathbb{R}^{n}$ they are linearly dependent over $\mathbb{R}$. Thus there exist real numbers $t_{0},...,t_{n}$, not all zero, such that
\begin{equation}
\sum_{j=0}^{n}t_{j}v_{j}=0\in\mathbb{R}^{n}.
\end{equation}
Equivalently, for each $i\in\{1,...,n\}$,
\begin{equation}
\sum_{j=0}^{n}t_{j}h_{i}(b_{j})=0.
\label{zero:sum}
\end{equation}
Fix $M>0$ so large that $\max_{0\le j\le n}|Mt_{j}|>2$. For each $i\in\{1,...,n\}$ define
\begin{equation} 
A_{i}:=\sum_{j=0}^{n}|h_{i}(b_{j})|\in[0,\infty). \end{equation} 
If $A_{i}=0$ then $h_{i}(b_{j})=0$ for all $j$ and the $i$-th constraint will automatically be satisfied. Otherwise $A_{i}>0$. Set
\begin{equation}
\eta:=\min\left\{1,\min_{1\le i\le n, A_{i}>0}\frac{\delta_{i}}{2A_{i}}\right\}>0.
\label{def:eta} 
\end{equation}
Choose rationals $q_{0},...,q_{n}\in\mathbb{Q}$ such that
\begin{equation}
|q_{j}-Mt_{j}|<\eta \quad (0\le j\le n).
\label{eta}
\end{equation}
(This is possible since $\mathbb{Q}$ is dense in $\mathbb{R}$.) Define
\begin{equation}
x:=\sum_{j=0}^{n}q_{j}b_{j}\in\mathbb{R}.
\label{ham:exp2}
\end{equation}

\begin{lemma} 
$x\in U.$ 
\end{lemma} 
\begin{proof} 
Indeed, fix $i\in\{1,...,n\}$. Using $\mathbb{Q}$-linearity of $h_{i}$ and~(\ref{zero:sum}), we compute
\[ h_{i}(x)=\sum_{j=0}^{n}q_{j}h_{i}(b_{j})=\sum_{j=0}^{n}(q_{j}-Mt_{j})h_{i}(b_{j})+M\sum_{j=0}^{n}t_{j}h_{i}(b_{j})=\sum_{j=0}^{n}(q_{j}-Mt_{j})h_{i}(b_{j}). \]
Hence, by the triangle inequality,
\begin{equation} 
|h_{i}(x)|\le\sum_{j=0}^{n}|q_{j}-Mt_{j}||h_{i}(b_{j})|<\eta\sum_{j=0}^{n}|h_{i}(b_{j})|=\eta A_{i}. \end{equation} 
If $A_{i}=0$ this gives $|h_{i}(x)|=0<\delta_{i}$. If $A_{i}>0$ then by~(\ref{def:eta}) we have $\eta A_{i}\le\delta_{i}/2<\delta_{i}$. Thus $|h_{i}(x)|<\delta_{i}$ for all $i$, so $x\in U$.
\end{proof} 

\begin{lemma} 
$f(x)>1$. 
\end{lemma} 
\begin{proof} 
In the Hamel expansion~(\ref{ham:exp2}) the coefficient of $b_{j}$ is exactly $q_{j}$. Therefore $f(x)=\max_{0\le j\le n}|q_{j}|$. Choose $j^{*}$ with $|Mt_{j^{*}}|=\max_{j}|Mt_{j}|>2$. Then by~(\ref{eta}) and $\eta\le1$,
\begin{equation} 
|q_{j^{*}}|\ge|Mt_{j^{*}}|-|q_{j^{*}}-Mt_{j^{*}}|>2-1=1. 
\end{equation} 
Hence $f(x)\ge|q_{j^{*}}|>1$.
\end{proof} 
Since $U$ was an arbitrary basic $\tau_{H}$-neighborhood of $0$, we have shown: for every basic neighborhood $U\ni0$ there exists $x\in U$ with $f(x)>1$. Because such basic neighborhoods form a neighborhood base at $0$, it follows that for every $\tau_{H}$-neighborhood $V\ni0$ there exists a basic $U$ with $0\in U\subseteq V$ and hence there exists $x\in V$ with $f(x)>1$. But $f(0)=0$ so $f$ cannot be $\tau_{H}$-continuous at $0$.
\end{proof}


The next result shows, on the other hand, that an arbitrary $\tau_{\Q}$-open set can be realized as preimages of a standard open set by bilaterally $\Q$-continuous functions:

\begin{theorem}
\label{thm:5.13}
Assume AC. 
Let $U\subseteq\mathbb{R}$ be $\tau_{\Q}$-open.
\begin{enumerate}
    \item[(1)] There exists a bilaterally $\mathbb{Q}$-continuous function $h:\mathbb{R}\rightarrow[0,\infty)$ such that $h^{-1}((0,\infty))=U$.
    \item[(2)] More generally, if $D\subsetneq\mathbb{R}$ is a nonempty open set in the usual topology, then there exists a bilaterally $\mathbb{Q}$-continuous function $g:\mathbb{R}\rightarrow\mathbb{R}$ such that $g^{-1}(D)=U$.
\end{enumerate}
\end{theorem}

\begin{proof}
We first prove (1). By Theorem~\ref{well-ordering-lines}, the family of rational affine lines in $\mathbb{R}$ (which has cardinality $\mathfrak{c}$) admits a  well-ordering indexed by ordinals $\alpha<\mathfrak{c}$,  as $(L_{\alpha})_{\alpha<\mathfrak{c}}$ so that for every $\alpha<\mathfrak{c}$, $F_{\alpha}:=L_{\alpha}\cap\bigcup_{\beta<\alpha}L_{\beta}$ is finite. We shall define, by transfinite recursion on $\alpha<\mathfrak{c}$, a partial function $h$ on $\mathbb{R}$ together with, for each $\alpha$, a continuous function $H_{\alpha}\in C(\mathbb{R})$ such that the following properties hold:
\begin{enumerate}
    \item[(i)] $H_{\alpha}|_{L_{\alpha}}=h|_{L_{\alpha}}$;
    \item[(ii)] for every $x\in L_{\alpha}$, $H_{\alpha}(x)>0\iff x\in U\cap L_{\alpha}$;
    \item[(iii)] if $x\in F_{\alpha}$, then the value assigned by $H_{\alpha}$ at $x$ coincides with the value of $h(x)$ already fixed at earlier stages.
\end{enumerate}
In addition, we maintain the following inductive invariant:
\begin{enumerate}
    \item[(iv)] whenever $x$ has already been assigned a value before stage $\alpha$, one has $h(x)>0\iff x\in U$.
\end{enumerate}

Assume inductively that for some $\alpha<\mathfrak{c}$ the function $h$ has already been defined on $\bigcup_{\beta<\alpha}L_{\beta}$, that for each $\beta<\alpha$ the witness $H_{\beta}$ has been constructed, and that properties (i)-(iv) hold at all earlier stages. We explain how to define $H_{\alpha}$ and extend $h$ to the new points of $L_{\alpha}$. Choose $a_{\alpha}\in\mathbb{R}$ and $d_{\alpha}\in\mathbb{R}\setminus\{0\}$ such that $L_{\alpha}=a_{\alpha}+\mathbb{Q}d_{\alpha}$. Thus the map $\varphi_{\alpha}:\mathbb{Q}\rightarrow L_{\alpha}$, $\varphi_{\alpha}(q)=a_{\alpha}+qd_{\alpha}$ is a bijection.

Since $U$ is $\tau_{\Q}$-open, 
$U\cap L_{\alpha}$ is open in the usual subspace topology on $L_{\alpha}$: this can be seen by applying Definition~\ref{def:tauq} in the two directions $d_{\alpha}$ and $-d_{\alpha}$. Map $\varphi_{\alpha}$ is affine, so a homeomorphism between  $\Q$ and $L_{\alpha}$. We infer that  
$O_{\alpha}:=\varphi_{\alpha}^{-1}(U\cap L_{\alpha})\subseteq\mathbb{Q}$ is open in $\mathbb{Q}$. 

Since every set open in $Q$ is of the form $W\cap \Q$ for some $W$ open in $\tau_{std}$, there exists an open set $W_{\alpha}\subseteq\mathbb{R}$ such that $W_{\alpha}\cap\mathbb{Q}=O_{\alpha}$. We now choose a continuous nonnegative function $k_{\alpha}:\mathbb{R}\rightarrow[0,\infty)$ whose positivity set is exactly $W_{\alpha}$. If $W_{\alpha}\ne\mathbb{R}$, we may take $k_{\alpha}(t):=\operatorname{dist}(t,\mathbb{R}\setminus W_{\alpha})$. If $W_{\alpha}=\mathbb{R}$ we simply put $k_{\alpha}\equiv1$. In either case, $k_{\alpha}\in C(\mathbb{R})$, $k_{\alpha}(t)\ge0$, and $k_{\alpha}(t)>0\iff t\in W_{\alpha}$.

Next consider the finite set $P_{\alpha}:=\varphi_{\alpha}^{-1}(F_{\alpha}\cap U)\subseteq\mathbb{Q}$. We claim that for every $t\in P_{\alpha}$, the number $h(\varphi_{\alpha}(t))$ is already defined and strictly positive. Indeed, if $t\in P_{\alpha}$ then $\varphi_{\alpha}(t)\in F_{\alpha}\cap U$ so $\varphi_{\alpha}(t)$ lies on some earlier line $L_{\beta}$ with $\beta<\alpha$; hence its value was assigned at an earlier stage. Since $\varphi_{\alpha}(t)\in U$ the inductive invariant (iv) gives $h(\varphi_{\alpha}(t))>0$. Moreover, because $t\in P_{\alpha}\subseteq O_{\alpha}=W_{\alpha}\cap\mathbb{Q}$, we have $t\in W_{\alpha}$ hence $k_{\alpha}(t)>0$. Therefore for every $t\in P_{\alpha}$ the number $s_{t}:=\log\left(\frac{h(\varphi_{\alpha}(t))}{k_{\alpha}(t)}\right)$ is well-defined.

By polynomial interpolation on the finite set $P_{\alpha}$ (including the case $P_{\alpha}=\emptyset$) we infer that there exists a continuous function $l_{\alpha}:\mathbb{R}\rightarrow\mathbb{R}$ such that $l_{\alpha}(t)=s_{t}$ $(t\in P_{\alpha})$. Define $m_{\alpha}(t):=k_{\alpha}(t)e^{l_{\alpha}(t)}$ $(t\in\mathbb{R})$. Then $m_{\alpha}\in C(\mathbb{R})$ and $m_{\alpha}(t)\ge0$ for all $t$. Furthermore, $m_{\alpha}(t)>0\iff k_{\alpha}(t)>0\iff t\in W_{\alpha}$. Also, for every $t\in P_{\alpha}$, $m_{\alpha}(t)=k_{\alpha}(t)e^{s_{t}}=h(\varphi_{\alpha}(t))$.

Now define
\begin{equation} 
H_{\alpha}(x):=m_{\alpha}\left(\frac{x-a_{\alpha}}{d_{\alpha}}\right) \quad (x\in\mathbb{R}). \end{equation} 
Since $m_{\alpha}$ is continuous, so is $H_{\alpha}$. We verify property (ii). Let $x\in L_{\alpha}$. Then $x=a_{\alpha}+qd_{\alpha}$ for a unique $q\in\mathbb{Q}$, namely $q=(x-a_{\alpha})/d_{\alpha}$. Therefore
\begin{align*}
 H_{\alpha}(x)>0\iff m_{\alpha}(q)>0\iff q\in W_{\alpha}\iff q\in W_{\alpha}\cap\mathbb{Q}\iff \\ 
 q\in O_{\alpha}\iff x\in U\cap L_{\alpha}. \end{align*} 
Thus $H_{\alpha}(x)>0\iff x\in U\cap L_{\alpha}$ $(x\in L_{\alpha})$.

We next verify compatibility on the overlap $F_{\alpha}$. If $x\in F_{\alpha}\cap U$ then $x=\varphi_{\alpha}(t)$ for some $t\in P_{\alpha}$, and by construction $H_{\alpha}(x)=m_{\alpha}(t)=h(\varphi_{\alpha}(t))=h(x)$. If $x\in F_{\alpha}\setminus U$, write again $x=\varphi_{\alpha}(q)$ with $q\in\mathbb{Q}$. Since $x\notin U\cap L_{\alpha}$, we have $q\notin O_{\alpha}$ hence $q\notin W_{\alpha}$. Therefore $m_{\alpha}(q)=0$ so $H_{\alpha}(x)=0$. On the other hand, $x$ was assigned earlier and $x\notin U$ so by the inductive invariant (iv) we also have $h(x)=0$. Thus $H_{\alpha}(x)=h(x)$ $(x\in F_{\alpha}\setminus U)$. Hence $H_{\alpha}(x)=h(x)$ $(x\in F_{\alpha})$ so property (iii) holds.

We now extend $h$ to the new points of $L_{\alpha}$ by putting $h(x):=H_{\alpha}(x)$ for $x \in L_{\alpha} \setminus (\bigcup_{\beta<\alpha} L_{\beta})$, and leaving all previously assigned values unchanged. This completes the recursive step.

After the recursion is finished, the function $h$ is defined on all of $\mathbb{R}$, because every point of $\mathbb{R}$ lies on at least one rational affine line. By construction, for each $\alpha$, $H_{\alpha}|_{L_{\alpha}}=h|_{L_{\alpha}}$ and $H_{\alpha}\in C(\mathbb{R})$. Since every rational affine line is some $L_{\alpha}$, Theorem~\ref{character-qcont} implies that $h$ is bilaterally $\mathbb{Q}$-continuous. It remains to prove that $h^{-1}((0,\infty))=U$.

Fix $x\in\mathbb{R}$ and let $\alpha$ be the least ordinal such that $x\in L_{\alpha}$. Then $x$ receives its value for the first time at stage $\alpha$, namely $h(x)=H_{\alpha}(x)$. If $x\in U$ then by property (ii), $H_{\alpha}(x)>0$ hence $h(x)>0$. Conversely, if $x\notin U$, then again by property (ii) we have $H_{\alpha}(x) \not> 0$; but $H_{\alpha}(x)\ge0$ so necessarily $H_{\alpha}(x)=0$ and this value is never altered at later stages. Therefore $h(x)>0\iff x\in U$. This proves (1).

We now prove (2). Let $D\subsetneq\mathbb{R}$ be nonempty and open in the usual topology. Choose a connected component $I$ of $D$. Since $D\ne\mathbb{R}$ the interval $I$ is proper, so it has a boundary point $c\in\mathbb{R}\setminus D$. Choose also a point $d_{0}\in I$. Define a continuous function $\psi:\mathbb{R}\rightarrow\mathbb{R}$ by
\[ \psi(t) = \begin{cases} c, & t\le0, \\ \frac{1}{1+t}c+\frac{t}{1+t}d_{0}, & t>0. \end{cases} \]
Then $\psi(t)=c\notin D$ for $t\le0$, while for $t>0$ the value $\psi(t)$ lies in the open interval between $c$ and $d_{0}$, hence in the component $I\subseteq D$. Therefore
\[ \psi^{-1}(D)=(0,\infty). \]
Let $h$ be the function obtained in part (1), and define $g:=\psi\circ h$. For each $\alpha$, the function $\psi\circ H_{\alpha}$ is continuous on $\mathbb{R}$, and on $L_{\alpha}$ one has
\begin{equation} 
(\psi\circ H_{\alpha})(x)=\psi(h(x))=g(x). \end{equation} 
Hence Theorem~\ref{character-qcont} shows that $g$ is bilaterally $\mathbb{Q}$-continuous. Finally,
\begin{equation} 
g^{-1}(D)=h^{-1}(\psi^{-1}(D))=h^{-1}((0,\infty))=U. \end{equation} 
This proves (2).
\end{proof}

\begin{corollary}
\label{cor:5.14}
Assume AC. 
Let $\tau_{\mathrm{init}}^{\mathrm{bi}\mathbb{Q}}$ denote the initial topology on $\mathbb{R}$ generated by all bilaterally $\mathbb{Q}$-continuous functions $g:\mathbb{R}\rightarrow\mathbb{R}$. Then
\begin{equation} 
\tau_{\mathrm{init}}^{\mathrm{bi}\mathbb{Q}}=\tau_{\Q}. \end{equation} 
\end{corollary}
\begin{proof}
Every bilaterally $\mathbb{Q}$-continuous function is $\tau_{\Q}$-continuous by Theorem~\ref{thm:inclusions}, so $\tau_{\mathrm{init}}^{\mathrm{bi}\mathbb{Q}}\subseteq\tau_{\Q}.$ 
Conversely, by part (2) of the theorem, every $\tau_{\Q}$-open set is of the form $g^{-1}(D)$ for some bilaterally $\mathbb{Q}$-continuous $g$ and some usual open set $D$. Hence
$\tau_{\Q}\subseteq \tau_{\mathrm{init}}^{\mathrm{bi}\mathbb{Q}}.$ 
Therefore equality holds.
\end{proof}

\begin{corollary}
\label{cor:5.15}
Assume AC. 
Let $\tau_{\mathrm{init}}^{\mathbb{Q}}$ be the initial topology on $\mathbb{R}$ generated by all $\mathbb{Q}$-continuous functions $f:\mathbb{R}\rightarrow\mathbb{R}$. Then
\begin{equation} 
\tau_{\mathrm{init}}^{\mathbb{Q}}=\tau_{\Q}. \end{equation} 
Moreover, for every $\tau_{\Q}$-open set $U\subseteq\mathbb{R}$ and every nonempty proper Euclidean-open set $D\subsetneq\mathbb{R}$, there exists a $\mathbb{Q}$-continuous function $f:\mathbb{R}\rightarrow\mathbb{R}$ such that $f^{-1}(D)=U$.
\end{corollary}
\begin{proof}
Since every bilaterally $\mathbb{Q}$-continuous function is $\mathbb{Q}$-continuous, we have
$ \tau_{\mathrm{init}}^{\mathrm{bi}\mathbb{Q}}\subseteq\tau_{\mathrm{init}}^{\mathbb{Q}}. $
On the other hand, every $\mathbb{Q}$-continuous function is $\tau_{\Q}$-continuous, so
$ \tau_{\mathrm{init}}^{\mathbb{Q}}\subseteq\tau_{\Q}. $
By the preceding corollary,
$ \tau_{\mathrm{init}}^{\mathrm{bi}\mathbb{Q}}=\tau_{\Q}.$
Therefore
$\tau_{\Q}=\tau_{\mathrm{init}}^{\mathrm{bi}\mathbb{Q}}\subseteq\tau_{\mathrm{init}}^{\mathbb{Q}}\subseteq\tau_{\Q},$
and hence
$\tau_{\mathrm{init}}^{\mathbb{Q}}=\tau_{\Q}. $
Finally, the realization theorem provides, for every $\tau_{\Q}$-open $U$ and every nonempty proper Euclidean-open $D$, a bilaterally $\mathbb{Q}$-continuous function $g$ such that $g^{-1}(D)=U$. Since $g$ is in particular $\mathbb{Q}$-continuous, the final assertion follows.
\end{proof}

Finally, we prove

\begin{theorem} The following are true: 
\begin{itemize} 
\item[(a). ] Assuming a Hamel basis exists, we have $C_{bi\Q}\subsetneq C_{\Q}$; that is there is a function that is $\Q$-continuous but not bilaterally $\Q$-continuous. 
\item[(b). ] $C(\tau_{\Q})\subsetneq LAC$, that is there exists a function that is locally almost continuous but not $\tau_{\Q}$-continuous. 
\end{itemize} 
\end{theorem} 
\begin{proof} 
\begin{itemize} 
\item[(a). ] Choose an additive function $K$ with $K(1)=0$ and $K(\sqrt{2})=1$, and define 
\begin{equation} 
f(x)=\left\{\begin{array}{cc}
0 & x\leq 0, \\
K(x) & x>0. \\
\end{array} 
\right. 
\end{equation} 
The function $f$ is $\Q$-continuous: this is clear at points $x\neq 0$, and follows easily at $x=0$ as well, by considering one-sided neighborhoods: on intervals left of $x=0$ use the zero function; on intervals to the right use the additive function $K$.

However, $f$ is not bilaterally $\Q$-continuous at $x=0$: 
Assume, toward a contradiction, that it was, and let \(\varepsilon>0\) be a corresponding radius. Choose \(r\in\Q_{>0}\) with \(r\sqrt2<\varepsilon\). On the chord \([-r,r\sqrt2]\), write
\[
 x_\alpha=\alpha(-r)+(1-\alpha)r\sqrt2,
 \qquad \alpha\in\Q\cap[0,1].
\]
The unique zero of \(x_\alpha\) occurs at
\[
 \alpha_0=\frac{\sqrt2}{1+\sqrt2}\notin\Q.
\]
If \(\alpha<\alpha_0\), then \(x_\alpha>0\) and, by \(\Q\)-linearity of \(K\),
\[
 f(x_\alpha)=K(x_\alpha)
 =\alpha K(-r)+(1-\alpha)K(r\sqrt2)
 =(1-\alpha)r.
\]
If \(\alpha>\alpha_0\), then \(x_\alpha<0\) and \(f(x_\alpha)=0\). Rational values of \(\alpha\) approach \(\alpha_0\) from both sides, but the corresponding limits of \(f(x_\alpha)\) are \(r/(1+\sqrt2)\) and \(0\). Hence \(f|_{S(-r,r\sqrt2)}\) has no continuous extension to the chord, a contradiction.
 
\item[(b). ] The Dirichlet function $f=\mathbbm{1}_{\Q}$ is locally almost continuous: each of its level sets is dense in $\R$, so for every $x\in \R$ and every  neighborhood $V$ of $f(x)$, $f^{-1}(V)$ contains either $\Q$ or $\R\setminus \Q$. In either case 
the closure of $f^{-1}(V)$ is $\R$. 

However, $f$ is \emph{not} $\tau_{\Q}$-continuous: indeed consider $\alpha\in \Q$ and $h\not \in \Q$. For every $q\in \Q, q>0$, $\alpha+qh\not \in Q$, so $f(\alpha+qh)=0$, hence $f(\alpha + qh)$ does not converge (as $q\rightarrow 0,q\in \Q$) to $f(\alpha)=1$. 
\end{itemize} 
\end{proof} 

\begin{corollary} Assume AC. Then there is no topology \(\tau\) on \(\R\) such that \(C(\tau)=\CQ\), and no topology such that \(C(\tau)=\CbiQ\), where the codomain carries the Euclidean topology. 
\end{corollary} 

\begin{proof} 
 If all functions in one of these classes were \(\tau\)-continuous, then its initial topology \(\tau_{\Q
}\) would be contained in \(\tau\). Hence \(C(\tau_{\Q})\subseteq C(\tau)\), contradicting the strict inclusions \(\CbiQ\subsetneq\CQ\subsetneq C(\tau_{\Q})\).

\end{proof}

\section{Examples of $\Q$-continuous functions} 

No "exotic" notion of continuity, however well-behaved, is valuable unless it is able to capture interesting classes of (ideally previously studied) real functions. For space reasons we give only one such example, that shows that bilateral $\Q$-continuity has indeed this desired property: 

\begin{definition} A $0$-\emph{monomial} is a constant function. Given integer $n\geq 1$, a $n$-\emph{monomial} \cite{1} (see also sections 13.4 and 15.9 in \cite{12}, as well as \cite{47}) is a function $f:\R\rightarrow \R$ satisfying the \emph{Fréchet’s equation}
\begin{equation} 
\frac{1}{n!}\Delta_{h}^{n} f(x)=f(h)
\end{equation} 
A \emph{polynomial} is a finite sum of monomials.
\label{def:poly}
\end{definition} 

With this definition, we have: 
\begin{theorem}\label{thm:polynomials-bilateral-Q}
Every polynomial $f\colon\mathbb R\to\mathbb R$ (in the sense of Definition~\ref{def:poly}) is bilaterally
$\mathbb Q$-continuous. More precisely, for every $y<z$ there exists
an ordinary polynomial $G_{y,z}\colon\mathbb R\to\mathbb R$ such that
\begin{equation} 
    G_{y,z}(t)=f(t)
    \qquad\text{for every }t\in L(y,z).
\end{equation} 
\end{theorem}

\begin{proof}
Write
\begin{equation} 
    f=f_0+f_1+\cdots+f_n,
\end{equation} 
where $f_0$ is a constant and, for $k\geq1$, $f_k$ is a
$k$-monomial. By the standard representation theorem for monomials
(see, e.g., \cite[Theorem~15.9.2]{12}), for every
$k=1,\ldots,n$ there exists a symmetric $k$-additive function $F_k\colon\mathbb R^k\longrightarrow\mathbb R$
such that
\begin{equation} 
    f_k(x)=F_k(x,\ldots,x)
    \qquad(x\in\mathbb R).
\end{equation} 

Fix $y<z$. For $\lambda\in\mathbb R$, define
\begin{equation} 
    P_{y,z}(\lambda)
    :=
    f_0+
    \sum_{k=1}^n
    \sum_{j=0}^k
    \binom{k}{j}
    \lambda^j(1-\lambda)^{k-j}
    F_k(
       \underbrace{y,\ldots,y}_{j},
       \underbrace{z,\ldots,z}_{k-j}
    ).
\end{equation} 
This is an ordinary polynomial in $\lambda$. Define
\begin{equation} 
    G_{y,z}(t)
    :=
    P_{y,z}\left(\frac{z-t}{z-y}\right),
    \qquad t\in\mathbb R.
\end{equation} 
Thus $G_{y,z}$ is an ordinary polynomial, and in particular it is
continuous. We claim that $G_{y,z}$ agrees with $f$ on $L(y,z)$. 

Indeed, let
$t\in L(y,z)$. Then $t=qy+(1-q)z$
for some $q\in\mathbb Q$, hence $q=\frac{z-t}{z-y}.$ 

Since an additive function is $\mathbb Q$-linear, every
$k$-additive function is $\mathbb Q$-linear in each variable.
Consequently, for $k=1,\ldots,n$,
\begin{align*}
    f_k(t)
    &=
    F_k(qy+(1-q)z,\ldots,qy+(1-q)z)\\
    &=
    \sum_{j=0}^k
    \binom{k}{j}
    q^j(1-q)^{k-j}
    F_k(
       \underbrace{y,\ldots,y}_{j},
       \underbrace{z,\ldots,z}_{k-j}
    ),
\end{align*}
where symmetry of $F_k$ was used to collect equal terms. Summing over
$k$ gives
\begin{equation} 
    f(t)=P_{y,z}(q)
        =P_{y,z}\left(\frac{z-t}{z-y}\right)
        =G_{y,z}(t).
\end{equation} 
Hence $
    G_{y,z}|_{L(y,z)}=f|_{L(y,z)}.$ Since such a continuous witness exists for \emph{every} $y<z$, $f$ is bilaterally $\mathbb Q$-continuous
on $\mathbb R$.
\end{proof}

\section{Context, Conclusions, Further Topics}
\label{sec:8}

As mentioned, the questions addressed in this paper 
are situated in an intellectual context with a substantially more 
ambitious foundational significance. They relate to problems about \emph{nonstandard models of classical mechanics}, specifically of the \emph{parallelogram rule for the composition of vectors (forces)}. Our brief presentation below cannot do justice to this topic. For a historical and epistemological discussion of the long-standing problem
of justifying the parallelogram rule, see also 
\cite{46}, Section 3.3 of \cite{31}, and \cite{32,33}. 

Briefly, it was long felt (at least since Daniel Bernoulli) that the composition rule for vectors needed an axiomatic justification. D'Alembert gave a first formulation \cite{27}. Darboux \cite{28} noted that a continuity axiom was implicit in D'Alembert's formulation. Other significant early work in this direction includes Hilbert \cite{37}, Schimmack \cite{29,30} and Hamel \cite{36}. In particular Schimmack studied an axiomatic system with seven axioms for vector addition (denoted $\oplus$), Axiom VII requiring, for every fixed vectors $v_1,v_2$, the continuity of the resultant of $\lambda \cdot v_1 \oplus \mu \cdot v_2$ as scalars $\lambda,\mu$ vary continuously. 

He proved that Axiom VII is independent from the rest of the axioms. The proof relies on constructing a model of the system of axioms I-VI  in which axiom VII fails to hold. His construction is based on Hamel's discontinuous additive functions \cite{35}. 

It is interesting to note that Hamel's work on discontinuous additive functions \cite{35} took place in the more general context of his preoccupations for the axiomatization of mechanics \cite{34}. Specifically, Hamel's motivation for studying discontinuous additive function was precisely the question whether the continuity axiom in the axiomatization of vector addition is genuinely necessary.

Recently, Laczkovich \cite{32,33} greatly completed the picture of nonstandard models of vector addition. In \cite{32} he considered an operation $\oplus$ on vectors in $\R^n$, constraining this operation by the following three axioms: 
\begin{itemize} 
\item[-] Axiom A1 (Commutativity and Associativity)
\begin{equation} 
a\oplus b = b\oplus a, (a\oplus b)\oplus c = a\oplus (b\oplus c)
\end{equation} 
\item[-] Axiom A2 (Ordinary addition for collinear vectors): 
\begin{equation} 
a\oplus b = a+b
\end{equation} 
whenever $b$ is a scalar multiple of $a$. 
\item[-] Axiom A3 (Rotational covariance): 
\begin{equation} 
Aa\oplus Ab = A(a\oplus b) 
\end{equation} 
for all vectors $a,b$ and rotations $A$. 
\end{itemize} 

The striking theorem from \cite{32} is that for $n\geq 3$, these axioms completely classify all possible operations: every one of them arises from an automorphism $H$ of the additive group $(\R,+)$ (i.e. a bijective additive function).

Specifically, given a bijective additive function $H:\R\rightarrow \R$, first define $\phi_{H}:\R^n\rightarrow \R^n$ as follows: if $a\neq 0$, define
\begin{equation} 
\phi_{H}(a)=H(|a|)\cdot \frac{a}{|a|}. 
\end{equation} 
($\phi_{H}(0)=0$). 
Then define 
\begin{equation} 
a\oplus_{H} b = \phi_{H}^{-1}( \phi_{H}(a)+ \phi_{H}(b))
\end{equation}  

In other words: $\oplus_{H}$ first distorts vectors $a,b$ by applying dilations by a factor of $H(|a|)$ ($H(|b|)$) to them, then takes ordinary vector addition of the distorted versions, and obtains the final result by subjecting the sum to the inverse distortion $\phi_{H}^{-1}$. 

Laczkovich proves \cite{32} that for every bijective additive function $H$, $\oplus_{H}$ satisfies axioms A1-A3 \textbf{and conversely:} for every operation $\oplus$ satisfying A1-A3 there exists a bijective additive function $H$ such that $\oplus = \oplus_{H}$. 

Ordinary (physical) vector addition corresponds to taking $H$ to be a linear function, $H(x)=cx$, $c\neq 0$. However, Laczkovich's result raises the intriguing possibility of defining "exotic" models of mechanics based on arbitrary bijective additive functions $H$. In such models vector addition $\oplus_H$ ceases to be continuous. 

\begin{shaded} 
In the context we described above, the question we asked in the introduction of the paper can be rephrased as follows: \textbf{is there an exotic (but somewhat well-behaved) "residual" notion of continuity that is shared even by these pathological models of vector addition?} 
\end{shaded}

Note that, since $\oplus_{H}:\R^3\times \R^3\rightarrow \R^3$, $\oplus_{H}(a,b)=\phi_{H}^{-1}(\phi_{H}(a)+\phi_{H}(b))$ is a bivariate function on $\R^3$, our theory of weak continuity does not,  strictly speaking, solve the problem above. 

However, one first approach to making a function like $\oplus_{H}$ well-behaved is to require that functions $\phi_{H},\phi_{H}^{-1}$ be well-behaved themselves. Since $\phi_{H}$ works by applying a dilation by a factor of $H(\cdot)$ on every ray from the origin $\{\lambda\cdot u: \lambda\in \R\}$ (which maps bijectively to $\R$), one condition to insist upon is that $H$ be well-behaved (i.e. "weakly continuous") on $\R$. To make all the models well-behaved, this needs to happen for \emph{every} bijective additive function $H$.\footnote{Restricting the  weak continuity requirement to \emph{bijective} additive functions is not really a constraint: assuming AC, every additive function is the sum of two bijective additive functions \cite{40}. We want a notion of weak continuity that is closed under summation. So we need to require that every additive function is "weakly continuous".} 

Our results single out $\Q$-continuity and its bilateral version as the only two "residual" notions of continuity of univariate functions (among those considered) that are somewhat well-behaved. But how should we choose among these two notions in applications, including those to nonstandard models of the paralellogram rule?  

Clearly, a principled answer to this question depends on the success of the intended (and so far not yet developed) applications to the two problems we discussed: 
\begin{itemize} 
\item[-] the development of a theory of differentiation based on arbitrary (rather than continuous) additive functions.
\item[-] the study of "pathological" models of vector addition seen as multivariate functions. 
\end{itemize} 

Hence giving a definitive answer to this question is beyond the scope of the present paper, and deferred to subsequent work. 

On the other hand, \emph{on mathematical grounds only}, bilateral $\Q$-continuity seems to be slightly more well-behaved than $\Q$-continuity: 
\begin{itemize} 
\item[-] First, it is the more restrictive notion, hence seems "closer to physical reality" (whatever that means).   
\item[-] Second, the chordwise characterization of bilateral $\Q$-continuity in Theorem~\ref{character-qcont} seems somewhat more elegant than that of ordinary $\Q$-continuity. 
\item[-] Finally, the example from Theorem~\ref{thm:polynomials-bilateral-Q} shows that the notion of bilateral $\Q$-continuity can capture some previously studied (interesting) classes of "exotic" functions. 
\end{itemize} 

We hope that we have motivated the study of weak/pathological extensions of continuity, and showed that at least some of the resulting notions are somewhat well-behaved. Here is a (short, nonexhaustive) list of problems we would like to see solved: 
\begin{itemize} 
\item[-] The extension of the $\Q$-continuity theory in this paper to multivariate real functions and its application to models of vector addition. 
\item[-] The (envisioned) applications to weak differentiability. 
\item[-] A more systematic study of the reverse mathematics of physical principles (see \cite{38,39}):  the "exotic" world of Cauchy-Hamel continuity seems to require specific nonconstructive axioms to exist. More generally, the question of understanding what constraints do specific set-theoretic axioms impose on physical theories is a fascinating topic. 
\item[-] Start with the Darboux-Schimmack axiomatization and replace Schimmack's ordinary continuity axiom VII by an appropriate (bilateral) $\Q$-continuity axiom. We may ask: \emph{What composition laws  $\oplus$  survive?} There are several plausible outcomes:
\begin{enumerate} 
\item (bilateral) $\Q$-continuity still forces $\oplus$ to be ordinary vector addition, once combined with the other axioms;
\item it permits exactly the Hamel-type nonstandard models identified by Laczkovich. 
\item it permits exactly the Hamel-type nonstandard models, at the cost of slightly altering the definition of bilateral $\Q$-continuity. 
\item it produces an intermediate class of models.
\end{enumerate} 
While we would \emph{a priori} favor the second alternative, we have no good reasons (besides mathematical taste) to suppose it would really be so. 
\end{itemize}

\section{Appendix}

In this section we present a result that, while interesting, was ultimately deemed not to be central to our main argument: during our investigation we had considered an additional test. However, this test  ultimately proved inconclusive. Still, the result formally proving this fact (Theorem~\ref{thm:8.1} below) seems to not have been stated in the literature and is of independent interest, so we chose to include it here.

\subsection{An Inconclusive Test: Sums of additive and Darboux functions}

By a classical result due to Sierpiński \cite{17}, every real function is the sum of two functions with the Darboux property. On the other hand not every function is a sum of a continuous and a Darboux function: \cite{6} have shown that $\mathcal{C}+\mathcal{D}\subseteq\mathcal{U}$. 

A natural candidate for a test we had considered was whether every function is the sum of a $\mathbb{Q}$-continuous and a Darboux function, and similarly for the other weak continuity notions we considered. 

The answer turned out to be affirmative, and actually holds in an even stronger form, which disqualified this question as a distinguishing test for the weak notions of continuity:

\begin{theorem}
\label{thm:8.1}
Assume AC. Then for every function $f:\mathbb{R}\rightarrow\mathbb{R}$ there exist an additive function $a:\mathbb{R}\rightarrow\mathbb{R}$ and a strongly Darboux function $d\in \mathcal{D}^{*}(\mathbb{R})$ such that
\begin{equation} 
f=a+d. \end{equation} 
\end{theorem}

\begin{proof}
We will first give an explicit construction.  

Indeed, fix $f:\mathbb{R}\rightarrow\mathbb{R}$. Let $(J_{n})_{n\in\mathbb{N}}$ enumerate all open intervals with rational endpoints. For each $n$ choose a rational $r_{n}\in J_{n}$. Fix a bijection $(y_{\alpha})_{\alpha<\mathfrak{c}}$ of $\mathbb{R}$.

Let $\mathcal{B}$ be a Hamel basis of $\mathbb{R}$ over $\mathbb{Q}$ and assume $1\in\mathcal{B}$. Since $|\mathcal{B}|=\mathfrak{c}=|\mathbb{N}\times\mathfrak{c}|$, we may choose pairwise distinct
$e_{n,\alpha}\in\mathcal{B}\setminus\{1\} \quad (n\in\mathbb{N},\alpha<\mathfrak{c}).$

For each $(n, \alpha)$ choose a nonzero rational $q_{n,\alpha}\in\mathbb{Q}\setminus\{0\}$ so small that
\begin{equation} 
x_{n,\alpha}:=r_{n}+q_{n,\alpha}e_{n,\alpha}\in J_{n}. 
\end{equation} 
This is possible because $J_{n}$ is open and $q\cdot e_{n,\alpha}\rightarrow0$ as $q\rightarrow0$ through rationals.

Define $a(1)=0$ (hence $a(q)=0$ for all $q\in\mathbb{Q}$). For each $(n, \alpha)$ define
\begin{equation} 
a(e_{n,\alpha}):=\frac{f(x_{n,\alpha})-y_{\alpha}}{q_{n,\alpha}} \end{equation} 
and define $a(b)=0$ for all remaining $b\in\mathcal{B}\setminus(\{1\}\cup\{e_{n,\alpha}\})$. Extend $a$ uniquely to an additive function $\mathbb{R}\rightarrow\mathbb{R}$ by $\mathbb{Q}$-linearity.

Now let $d:=f-a$. For each $(n, \alpha)$ we have
$a(x_{n,\alpha})=a(r_{n})+a(q_{n,\alpha}e_{n,\alpha})=0+q_{n,\alpha}a(e_{n,\alpha})=f(x_{n,\alpha})-y_{\alpha}$, 
so $d(x_{n,\alpha})=y_{\alpha}$. Thus for each fixed $n$, $d(J_{n})=\mathbb{R}. $

Since every interval $I$ contains an open interval $J_n$ with rational endpoints 
\begin{equation} 
d(I)\supseteq d(J_n)=\R. 
\end{equation} 
Hence $d\in \mathcal{D}^{*}(\mathbb{R})$. Finally, $f=a+d$.
\end{proof}

\begin{remark}
\label{rem:8.2}
Related constructions 
involving Hamel and Darboux functions appear in the note of Rădulescu-Rădulescu \cite{16}.
\end{remark}

\begin{remark}
\label{rem:8.4}
We can rederive Theorem~\ref{thm:8.1} from results of Płotka in \cite{15}: he introduces, for classes $F_{1}, F_{2}\subseteq\mathbb{R}^{\mathbb{R}}$ the cardinal
\begin{equation} Add(F_{1},F_{2}):=\min\{|F|:F\subseteq\mathbb{R}^{\mathbb{R}}\ \&\ \neg\exists g\in F_{1}(g+F\subseteq F_{2})\}, \end{equation} 
and proves (Prop. 1(3)) the equivalence
\begin{equation}
Add(F_{1},F_{2})\ge2\iff\mathbb{R}^{\mathbb{R}}=F_{2}-F_{1}.
\label{12:6}
\end{equation}
He then characterizes (Thm. 10(i)) 
$Add(\AD,\mathcal{D})$; in particular $Add(\AD,\mathcal{D})\ge2$. By ~(\ref{12:6}) with $F_{1}=\AD$ and $F_{2}=\mathcal{D}$ we get $\mathbb{R}^{\mathbb{R}}=\mathcal{D}-\AD=\mathcal{D}+\AD$, 
so every $f$ is a sum of an additive function and a Darboux function.

Moreover, Płotka also works with strongly Darboux functions. In the proof of Lemma 27 (\cite{15}), he gives a strengthening of the previous result  (see the explicit ``More precisely...'' sentence in that paper) that can be applied to infer that every real function is the sum of an additive and a strongly Darboux function, in other words Theorem \ref{thm:8.1}.
\end{remark}

\begin{remark}
\label{rem:8.5}
Since the $\mathbb{Q}$-continuous functions are closed under addition, we have $\AD+\mathcal{C}(\R) \subseteq C_\mathbb{Q}+C_\mathbb{Q}=C_\mathbb{Q}$. 
Since there are functions outside $C_\mathbb{Q}$ (e.g. the function from Theorem~\ref{thm:5.10}) not every function is the sum of an additive and a continuous function.
\end{remark}

\medskip



\end{document}